\documentclass[10pt,reqno]{amsart}
\usepackage{geometry}                % See geometry.pdf to learn the layout options. There are lots.
\usepackage[ansinew]{inputenc}
\usepackage[UKenglish,american]{babel}
 \usepackage{import}
\usepackage{amsmath, amsfonts, amssymb}
\usepackage[usenames, dvipsnames]{color}
\usepackage{fullpage, ifthen}
\usepackage[all]{xy} \usepackage{hyperref}
\hypersetup{
    colorlinks=true, %set true if you want colored links
    linktoc=all,     %set to all if you want both sections and subsections linked
    linkcolor=blue,  %choose some color if you want links to stand out
    citecolor=JungleGreen,
}

\newcommand{\colorlabel}[1]{%
  \refstepcounter{equation}%
  \tag{\textcolor{#1}{\theequation}}}

\newtheorem{theo}{Theorem}[section]
\newtheorem{prop}[theo]{Proposition}
\newtheorem{lema}[theo]{Lemma}
\newtheorem{coro}[theo]{Corollary}
\theoremstyle{definition}
\newtheorem{defi}[theo]{Definition}
\theoremstyle{remark}
\newtheorem{obs}[theo]{Remarks}
\newtheorem{rem}[theo]{Remark}

\newcommand{\ext}{\Lambda}

\newcommand{\spi}{\Sigma}

\DeclareMathOperator{\Gl}{GL}
\DeclareMathOperator{\diag}{diag}
\DeclareMathOperator{\End}{End}
\DeclareMathOperator{\SO}{SO}
\DeclareMathOperator{\OO}{O}
\DeclareMathOperator{\Cl}{Cl}
\DeclareMathOperator{\GG2}{G_2}
\DeclareMathOperator{\Gr}{G}
\DeclareMathOperator{\S7}{Spin(7)}
\DeclareMathOperator{\Spin8}{Spin(8)}
\DeclareMathOperator{\Fix}{Stab}
\DeclareMathOperator{\grad}{grad}
\DeclareMathOperator{\sgn}{sgn}
\DeclareMathOperator{\Sym}{Sym_0^2}

\DeclareMathOperator{\7s}{\mathfrak{spin}(7)}
\DeclareMathOperator{\2g}{\mathfrak{g}_2}
\newcommand{\g}{\mathfrak{g}}
\newcommand{\h}{\mathfrak{h}}
\newcommand{\m}{\mathfrak{m}}

\newcommand{\KK}{\mathfrak{L}}
\newcommand{\rt}{\mathrm{T}}
\newcommand{\Id}{\mathrm{Id}}
\newcommand{\dt}{\frac{\partial}{\partial t}}
\newcommand{\Vol}{\nu_8}

\newcommand{\con}{\nabla^{'}}
\newcommand{\car}{\nabla^{c}}

\newcommand{\sv}{Q}

\newcommand{\tor}{\Gamma}
\newcommand{\alt}[1]{\mathrm{alt}({#1})} 

\newcommand{\C}{\mathbb{C}}
\newcommand{\R}{\mathbb{R}} 
 
\newcommand{\PP}{\mathbb{P}} 
\newcommand{\Z}{\mathbb{Z}} 

\newcommand{\mcal}[1]{\mathcal{#1}}
\newcommand{\sss}{\mcal{S}}
\newcommand{\aaa}{\mcal{A}}
\newcommand{\ttt}{\mcal{T}}
\newcommand{\hhh}{\mcal{H}}
\newcommand{\llll}{\mcal{L}}
\newcommand{\eee}{\mcal{E}}
\newcommand{\www}{\mcal{W}}
\newcommand{\WW}{\mcal{W}}

\newcommand{\ddd}{\mcal{D}}

\newcommand{\MaxAb}{\mcal{QA}}
\newcommand{\IMA}{\overline{\mcal{QA}}}

\newcommand{\cic}{\mathfrak{S}}
\newcommand{\proy}{p}

\newcommand{\pso}[1]{P({#1})}
\newcommand{\pspin}[1]{\tilde{P}({#1})}

\begin{document}
\title{Spinorial classification of Spin(7) structures}

\author{Luc\' ia Mart\' in-Merch\' an}

\address{Departamento de \' Algebra, Geometr\'{\i}a y Topolog\'{\i}a\\ Universidad Complutense de Madrid \\ Plaza de Ciencias 3 \\ 28040 Madrid \\ Spain}
\email{lmmerchan@ucm.es}
\begin{abstract}
We describe the different classes of $\S7$ structures in terms of spinorial equations. We relate them to the spinorial description of $\GG2$ structures in some geometrical situations. Our approach enables us to analyze invariant $\S7$ structures on quasi abelian Lie algebras.
\end{abstract}


\maketitle
% Upper part of the page. The '~' is needed because \\
% only works if a paragraph has started.
%\includegraphics[width=0.35\textwidth]{logotipo-uma-footer}~\\[1cm]
 %\tableofcontents{}
%\section{Introduction}

\section{Introduction} 

Berger's list \cite{BE} (1955) of possible holonomy groups of simply connected, irreducible and non-symmetric Riemannian manifolds contains the so-called exceptional holonomy groups, $\GG2$ and $\S7$, which occur in dimensions $7$ and $8$ respectively. Non-complete metrics with exceptional holonomy were given by Bryant in \cite{Br87}, complete metrics were obtained by Bryant and Salamon in \cite{BrS89}, but compact examples were not constructed until 1996, when Joyce published \cite{J1}, \cite{J2} and \cite{J3}. 

The remaining groups of Berger's list different from $\mathrm{SO(n)}$, called special holonomy groups, are $\mathrm{U(n)}$, $\mathrm{SU(n)}$, $\mathrm{Sp(n)}$ and $\mathrm{Sp(n)}\cdot \mathrm{Sp(1)}$. If the holonomy of a Riemannian manifold is contained in a group $\Gr$, the manifold admits a $\Gr$ structure, that is, a reduction  to $\Gr$ of its frame bundle. Therefore, holonomy is homotopically obstructed by the presence of $\Gr$ structures. Examples of manifolds endowed with $\Gr$ structures for some of the holonomy groups in the Berger list are not only easier to obtain than manifolds with holonomy in $\Gr$, but also relevant in M-theory, especially if they admit a characteristic connection \cite{FI}, that is, a metric connection with totally skew-symmetric torsion whose holonomy is contained in $\Gr$. It is worth mentioning that Ivanov proved in \cite{IV04} that each manifold with a $\S7$ structure admits a unique characteristic connection. Moreover, Friedrich proved in \cite{FR03} that $\S7$ is the unique compact simple Lie group $\mathrm{G}$ such that all the $\mathrm{G}$ structures admit a unique characteristic connection.

The Lie group $\GG2$ is compact, simple and simply connected. It consists of the endomorpisms of $\R^7$ which preserve the cross product from the imaginary part of the octonions \cite{SW10}. Hence, a $\GG2$ structure on a manifold $\sv$ determines a $3$-form $\Psi$, a metric and an orientation. In \cite{MG}, Fern\'{a}ndez and Gray classify $\GG2$ structures into $16$ different classes in terms of the $\GG2$ irreducible components of $\nabla \Psi$. Related to this, the analysis of the intrinsic torsion in \cite{SC} allowed to obtain equations involving $d\Psi$ and $d(*\Psi)$ for each of the $16$ classes, determined by the $\GG2$ irreducible components of $\ext^4 T^*\sv$ and $\ext^5 T^* \sv$. In particular, one obtains that the holonomy of $Q$ is contained in $\GG2$ if and only if $d\Psi=0$ and $d(*\Psi)=0$.
%Among those classes calibrated and cocalibrated structures, with $d\Psi=0$ and $d(*\Psi)=0$ respectively, have been studied specially in the setting of nilmanifolds. The list of nilmanifolds admiting calibrated and cocalibrated invariant structures can be found in [], []. MAS???. 
The Lie group $\S7$ is also compact, simple and simply connected. It is the group of endomorphisms of $\R^8$ which preserve the triple cross product from the octonions \cite{SW10}. Thus, a $\S7$ structure on a manifold $M$ determines $4$-form $\Omega$, a metric and an orientation. In \cite{MF86}, Fern\'{a}ndez classifies $\S7$ structures into $4$ classes in terms of differential equations for $d\Omega$, which are determined by the $\S7$-irreducible components of $\ext^5 T^*M$. Parallel structures verify $d\Omega=0$, locally conformally parallel structures satisfy $d\Omega= \theta \wedge \Omega$ for a closed $1$-form $\theta$ and balanced structures verify $*(d\Omega)\wedge \Omega = 0$. A generic $\S7$ structure, which does not satisfy any of the previous conditions, is called mixed. %An important aspect of $\S7$ structures, which distinguish them from other types of geometries, is that they admit a unique characteristic connection, as it was proved by Ivanov in \cite{IV04}. This work also shows that the Ricci curvature of the manifold is expressed in terms of the torsion of the characteristic connection and $\Omega$.

The relationship between $\GG2$ and $\S7$ structures was firstly explored by Mart\'{i}n-Cabrera in \cite{C1}. Each oriented hypersurface of a manifold equipped with a $\S7$ structure naturally inherits a $\GG2$ structure whose type is determined by the $\S7$ structure of the ambient manifold and some extrinsic information of the submanifold, such as the Weingarten operator. %For instance, each hypersurface of a parallel $\S7$ manifold admits a cocalibrated $\GG2$ structure. %which is also a locally conformally parallel if the hypersurface is totally umbilic. 
Following the same viewpoint, Mart\'{i}n-Cabrera constructed  $\S7$ structures on $S^1$-principal bundles over $\GG2$ manifolds in \cite{C2}. Both approaches  allowed to construct manifolds with $\GG2$ and $\S7$ structures of different pure types.

It turns out that manifolds admitting $\mathrm{SU(3)}$, $\GG2$ and $\S7$ structures are spin and their spinorial bundle has a unitary section $\eta$ which determines the structure. In \cite{ACFH15}, spinorial formalism was used to deal with the distinct aspects of $\mathrm{SU(3)}$ and $\GG2$ structures, such as the classification of both types of structures, $\mathrm{SU(3)}$ structures on hypersurfaces of $\GG2$ manifolds and different types of Killing spinors. A clear advantage of this viewpoint is that a unique object, the spinor, encodes the whole geometry of the structure. For instance, a $\GG2$ structure on a Riemannian manifold $(Q,g)$ with associated $3$-form $\Psi$ is determined by a suitable spinor $\eta$ according to the formula $\Psi(X,Y,Z)=(X\eta,YZ\eta)$ where $(\cdot,\cdot)$ denotes the scalar product in the spinorial bundle and juxtaposition of vectors indicates the Clifford product. Any oriented hypersurface $\sv'$ with normal vector field $N$ inherits an $\mathrm{SU(3)}$ structure implicitly defined by $\Psi= N^*\wedge \omega + \mathrm{Re}(\Theta)$, where $N^*(X)=g(N,X)$ for $X\in T\sv$. But both the Kahler form $\omega$ and the $(3,0)$-form $\mathrm{Re(\Theta)}$ turn out to be determined by the same spinor $\eta$. 

In this paper we follow the ideas of \cite{ACFH15} to describe the geometry of $\S7$ structures from a spinorial viewpoint, starting from the classification of these structures, continuing to analyze the relationship between $\GG2$ and $\S7$ structures and finishing with the study of invariant $\S7$ structures on quasi abelian Lie algebras.

Our first result, Theorem \ref{ppal} in section \ref{s}, describes each type of $\S7$ structure with spinorial equations. To state it, we have to mention that if the structure is determined by a spinor $\eta$ and $R$ is a $\S7$ reduction of the frame bundle of the manifold, there is a natural isomorphism $c\colon R \times_{\S7} \7s^ \perp \to \langle \eta \rangle^ \perp$ (see Lemma \ref{torin}, section \ref{tor} for details).
\begin{theo}  Let $D$ be the Dirac operator of the spinorial bundle and take $V\in TM$ such that $D\eta=V\eta$. The $\S7$ structure $\Omega$ defined by $\eta$ is:
\begin{enumerate}
\item[1.] Parallel if $\nabla \eta=0$.
\item[2.] Locally conformally parallel if $i(V)\Omega= 28\, \alt{c^{-1}\nabla \eta}$.
\item[3.] Balanced if $D\eta=0$.
\end{enumerate}
\end{theo}

Our techniques also allow us to identify the intrinsic torsion of the structure and to obtain the formula for the unique characteristic connection of each $\S7$ structure, given by Ivanov in \cite[Theorem 1.1]{IV04}. 

We also introduce the concept of $\GG2$ distributions, a general setting to relate $\GG2$ and $\S7$ structures.
\begin{defi}
Let $(M,g)$ be an oriented $8$-dimensional Riemannian manifold  and let $\ddd$ be a cooriented distribution of codimension $1$. We say that $\ddd$ has a $\GG2$ structure if the principal $\SO (7)$ bundle $\pso{\ddd}$ is spin and the spinorial bundle $\spi(\ddd)$ admits a unitary section.
\end{defi}

This construction allows us to obtain the results which appear in \cite{C1} and \cite{C2} about $\GG2$ structures on hypersurfaces of $\S7$ manifolds and  $S^1$-principal bundles over $\GG2$ manifolds. Related to this, we also study warped products of manifolds admitting a $\GG2$ structure with $\R$. 
%In Theorem \ref{G2type} we relate precisely how is the type of $\GG2$ structure is restricted by type of the $\S7$ structure of the ambient manifold and extrinsic information of the submanifold. Moreover, Corollary \ref{cprin} states how the curvature of a $S^1$-principal bundle and the type of $\GG2$ structure of the base manifold restricts the type of the $\SS7$ structure of the bundle.

The formalism of $\GG2$ distributions enables us to study invariant $\S7$ structures on quasi-abelian Lie algebras, that is, Lie algebras with a codimension $1$ abelian ideal. To state the result, which is Theorem \ref{qab}, suppose that the Lie algebra is $\g=\langle e_0,\dots, e_7 \rangle$ with abelian ideal $\R^7=\langle e_1,\dots, e_7 \rangle$ and it is endowed with the canonical metric and volume form.
%The situation corresponds to a  simply connected Lie group endowed with an invariant $\S7$ structure and an invariant integrable distribution, which inherits a parallel $\GG2$ structure. We obtain the following theorem: 

\begin{theo} \label{qab}  Denote by $\eee=ad(e_0)|_{\R^7}$ and let $\eee_{13}$ and $\eee_{24}$ be the symmetric and skew-symmetric parts of the endomorphism. Then, $\g$ admits a $\S7$ structure of type:
\begin{enumerate}
\item[1.]Parallel, if and only if $\eee_{13}=0$ and the eigenvalues of  $\eee_{24}$ are $0,\pm \lambda_1 i, \pm \lambda_2 i, \pm( \lambda_{1}+\lambda_{2})i$, for some $0\leq \lambda_1\leq \lambda_2$.
\item[2.] Locally conformally parallel and non-parallel if and only if $\eee_{13}=h \, \Id$ with $h\neq 0$ and the eigenvalues of $\eee_{24}$ are $0,\pm \lambda_1 i, \pm \lambda_2 i, \pm( \lambda_{1}+\lambda_{2})i$, for some $0\leq \lambda_1\leq \lambda_2$.
\item[3.] Balanced if and only if $\g$ is unimodular and the eigenvalues of $\eee_{24}$ are $0,\pm \lambda_1 i, \pm \lambda_2 i, \pm( \lambda_{1}+\lambda_{2})i$, for some $0\leq \lambda_1\leq \lambda_2$.
\end{enumerate}
Moreover, if $\eee_{24}\neq 0$ then it admits a $\S7$ structure of mixed type.
\end{theo}

From this, it follows (Corollary \ref{solv}) that there are no quasi abelian solvmanifolds which admit a locally conformally parallel $\S7$ structure. In addition, this result allows us to give an example of a nilmanifold admitting both an invariant balanced structure and an invariant mixed structure.  A compact manifold admitting a parallel structure is also obtained as a quotient of a simply connected Lie group whose Lie algebra is quasi abelian. Despite not being diffeomorphic to a torus, it is flat. Indeed, we prove that quasi abelian Lie algebras which admit an invariant $\S7$ parallel structure are flat (Corollary \ref{flat}).

This paper is organized as follows. Section \ref{pre} contains a review of algebraic aspects of $\S7$ geometry. Section \ref{tor} identifies the instrinsic torsion of the Levi-Civita connection with a spinor, section \ref{s} is devoted to obtain the classification of $\S7$ structures in terms of spinors and section \ref{car} provides an alternative proof of the existence of the characteristic connection. Section \ref{g2} provides a complete analysis of $\GG2$ structures on distributions and then focuses on the particular cases described above. Finally, section \ref{sqab} deals with invariant structures on quasi abelian Lie algebras and provides compact examples.

\section{Preliminaries} \label{pre}
In this section we introduce some aspects of Clifford algebras, $8$-dimensional spin manifolds and $\S7$ representations, which can be found in  \cite{FR00}, \cite{LM89} and \cite{SW10} as well as the notations that we will use in the sequel.

\subsection{$\S7$ structures.} 

Let $(M,g)$ be an oriented Riemannian $8$-manifold and let $\pso{M}$ be the associated $\SO(8)$ frame bundle.   
Provided that $M$ is spin, that is $w_2(M)=0$, we can take a $\Spin8$ principal bundle $\pspin{M}$ over $M$ which is a double covering $\pi \colon \pspin{M} \to \pso{M}$ equivariant under the adjoint action $Ad\colon \Spin8 \to \SO(8)$. We may also denote by $Ad$ the induced action of $\Spin8$ on $TM$.

 %We shall then identify $\tilde{P}_p \equiv \Spin8$, $P_p \equiv \SO(8)$ and $\pi_p=Ad$ at each $p\in M$.

The associated spinorial bundle is $\spi(M)= \pspin{M} \times_{\rho}\R^{16}$ where we have denoted by $\rho \colon \Spin8 \to \SO(16)$ the real spinorial representation, constructed by restricting the isomorphism $\Cl_8\cong \Gl(16)$ and equipping $\R^{16}$ with a metric $(\cdot,\cdot)$ which makes the Clifford product a skew-symmetric endomorphism. 
The induced metric on $\spi(M)$ will be denoted in the same way, and the elements of this bundle by $\phi=[\tilde{F},v]$, where $\tilde{F}\in \pspin{M}$ and $v\in \R^{16}$.

The Clifford multiplication with a vector field is extended to an action of $\ext^k T^*M$  defined as follows. 
\begin{enumerate}
\item[1.] The product with a covector is defined by $X^*\phi = X\phi$, where we used the canonical identification  between the tangent and the cotangent bundle: $X^*=g(X,\cdot)$.
\item[2.] If the product is defined on $\ext^\ell T^*M$ when $\ell\leq k$, we define 
\[(X^*\wedge \beta) \phi = X (\beta \phi) + (i(X)\beta) \phi,\]
where $i(X)\beta $ denotes the contraction, $\beta \in \ext^kT^*M$ and $X\in TM$. This product is extended lineary to $\ext ^{k+1}T^*M$.

\end{enumerate}  For instance, we have:
 \begin{align} \colorlabel{blue}
& (X^*\wedge Y^*)\phi = (XY+ g(X,Y))\phi, \\
\colorlabel{blue} &( X^*\wedge Y^* \wedge Z^*)\phi = (XYZ +  g(X,Y)Z - g(X,Z)Y + g(Y,Z)X)\phi \label{eqn2}. 
\end{align}

The volume form $\Vol$ of $\R^8$ provides $\R^{16}$ with a $\Spin8$ equivariant endomorphism:
$$\Vol \cdot  \colon  \R^{16} \to \R^{16}, \quad \phi \longmapsto \Vol\phi.$$
Since $\Vol^2=1$, there is a splitting $\R^{16}=\Delta^+ \oplus \Delta^-$ where $\Delta^{\pm}$ is the eigenspace associated to $\pm 1$.
Therefore, $\spi(M)=\spi(M)^+\oplus \spi(M)^-$, where $\spi(M)^{\pm}= \pspin{M}\times_\rho \Delta^{\pm}$. Also note that $X ( \spi(M)^{\pm} ) \subset \spi(M)^{\mp}$ if $X\in \mathfrak{X}(M)$. 

At each $p\in M$, the action $\Spin8 \to \SO(\Sigma_p(M)^+)$, $\tilde{\varphi} [\tilde{F},v] = [\tilde{F}, \rho(\tilde{\varphi})v] $ is a double covering, so that the existence of a unitary spinor $\eta \in \Gamma(\spi(M)^+)$ determines an identification between $\S7$ and the stabilizer of $\eta_p$, $\Fix(\eta_p)$. Besides, the restriction $Ad\colon \Spin8 \to \SO(T_pM)$ to $\Fix(\eta_p)$ is injective since $\ker(Ad)=\{1,-1\}$ and $-1\notin \Fix(\eta_p)$. 
% Moreover, if we denote the adjoint representation by $Ad\colon \Spin8 \to \SO(8)$, then $\rho(g)(X\phi)=Ad(g)(X)\rho(g)(\phi)$ if $g\in \Spin8$, $X\in \R^{8}$, $\phi \in \R^{16}$.

The previous considerations allow us to define a $4$-form $\Omega$ on $M$ such that $Ad(\Fix(\eta_p))= \Fix(\Omega_p)$. Indeed, observe that there is a well defined map:
\[ TM\times TM \times TM \to TM, \quad (X,Y,Z) \longmapsto X\times Y \times Z  \mbox{ s.t, }   (X\times Y \times Z) \eta= (X^*\wedge Y^* \wedge Z^*) \eta, 
\]
%\begin{align*} TM\times TM \times TM \to TM, \quad (X,Y,Z) \longmapsto X\times Y \times Z  \colon & \\
%& (X^*\wedge Y^* \wedge Z^*) \eta = (X\times Y \times Z) \eta, 
%\end{align*}
which turns out to be a positive triple product, that is, it verifies \cite[Definition 6.1]{SW10}:

\begin{enumerate}
\item[1.] The vector $X \times Y \times Z$ is perpendicular to $X$, $Y$ and $Z$.
\item[2.] $\| X\times Y \times Z\| = \|X^*\wedge Y^* \wedge Z^*\|$.
\item[3.] If we take orthonormal vectors $W,X,Y,Z$ such that $W$ is perpendicular to $X \times Y \times Z$, then $X \times Y \times (X \times Z \times W)= Y \times Z \times W$.
\end{enumerate}

The first property follows from (\ref{eqn2}) and the second one is obvious. To check the third one we observe that $Y$ is perpendicular to $X\times Z \times W$ since 
	$g(W,X\times Y \times Z)=(W\eta, XYZ\eta)=(Y\eta,XZW\eta)=g(Y,X\times Z\times W),$
	and therefore: 
	 \[ X\times Y \times (X \times Z \times W) \eta = XYXZW\eta =YZW\eta= (Y\times Z \times W) \eta .\]

 \begin{defi} The associated $4$-form to the triple cross product is: \label{spinform} 
\begin{align*}
\Omega(W,X,Y,Z) & = ((X\times Y \times Z)\eta, W \eta)= ((XYZ + g(X,Y)Z - g(X,Z)Y + g(X,Y)Z)\eta,W\eta) \\
\quad &= \frac{1}{2} ((-WXYZ + WZYX)\eta,\eta).
\end{align*}
\end{defi}	 

Since $\tilde{\varphi}(X\phi)=Ad(\tilde{\varphi})(X)(\tilde{\varphi}\phi)$ if $\tilde{\varphi}\in \Spin8$, $X\in TM$ and $\phi \in \spi(M)$, it is not hard to check that $\Fix(\eta_p)= \Fix (\Omega_p)$.
Some important properties of this form are the following:
\begin{enumerate}
\item[1.] If $(X_0,\dots, X_7)$ is an orthonormal oriented basis and $\sigma$ is a permutation then $*\Omega = \Omega$ since $ X_{\sigma(0)}X_{\sigma(1)}X_{\sigma(2)}X_{\sigma(3)} \eta = (-1)^{\sgn(\sigma)}  X_{\sigma(4)}X_{\sigma(5)}X_{\sigma(6)}X_{\sigma(7)} \eta$. 
\item [2.] Given orthonormal vector fields $e_0,e_1,e_2,e_4$ such that $e_4$ is perpendicular to $e_3= e_0\times e_1 \times e_2$, we can find \cite[Theorem 7.12]{SW10} an orthonormal frame $(e_0,\dots,e_7)$ such that:
	\begin{align*} \colorlabel{blue} \label{spinform}
	\Omega = &  \quad e^{0123}-e^{0145}-e^{0167}-e^{0246}+  e^{0257}- e^{0347} - e^{0356} \\ &  + e^{4567} - e^{2367} - e^{2345} - e^{1357} + e^{1346} - e^{1256} - e^{1247},
	\end{align*} 
where we have used the short-hand notation  $e^i$ for $g(e_i,\cdot)$ and $e^{ijkl}$ for $e^i\wedge e^j\wedge e^k \wedge e^l$. We will also denote the Clifford product $e_i e_j$ by $e_{ij}$ and so on. A frame of this type will be called a Cayley frame. Since those frames verify 
$(e_0\cdots e_7)\eta=\eta$, they are positively oriented.
\end{enumerate}

\subsection{$\S7$ representations.}

Let us denote the standard basis of $\R^8$ by $(e_0,\dots, e_7)$, and the standard $\S7$ structure of $\R^8$ by $\Omega_0$, given by (\ref{spinform}).

The canonical representation of $\S7=\Fix (\Omega_0) \subset \SO(8)$ on $\ext^k \R^8$ induces an orthogonal decomposition of this space into irreducible $\S7$ invariant subspaces. The expression $\ext_l^k \R^8$ denotes such an $l$-dimensional subspace of $\ext^k \R^8$. Observe that Hodge star operator $*$ gives isomorphisms between $\ext^k \R^8$ and $\ext^{8-k}\R^8$ determining that $\ext_l^{k}\R^8=*\ext_l^{8-k}\R^8$ if $k\leq 4$.
We are going to describe briefly the splitting at degrees $k=2$ and $k=3$ but a complete proof can be found in \cite[Theorem 9.8]{SW10}. The decomposition goes as follows:
\begin{align*}
\ext^2 \R^8 =& \ext^2_7 \R^8 \oplus \ext^2_{21} \R^8, \\
\ext^3 \R^8 = & \ext^3_8 \R^8 \oplus  \ext^3_{48} \R^8.
\end{align*}

The first one comes from the orthogonal splitting $\ext^2 \R^8 = \mathfrak{so}(8) = \mathfrak{spin}(7) \oplus \mathfrak{m}$, where $\m=\7s^{\perp}$. An alternative description may be done by considering the map:
 \[  \ext^2 \R^8  \to \ext^2 \R^8, \quad \beta \longmapsto *(\beta\wedge\Omega_0),\]
which is $\S7$-equivariant, symmetric and traceless. Therefore, $\ext^2 T^*M$ splits into eigenspaces which must coincide with the previous ones due to the irreducibility. It can be checked that the eigenvalues are $3$ on $\ext^2_7 \R^8$ and $-1$ on $\ext^2_{21} \R^8$.  Moreover, the set $\{ \alpha_j= e^{0j} - i(e_0)i(e_j)\Omega_0\}_{j=1}^{7}$ is a basis of $\ext_7^2T^*M$. 

The subspaces involved in the second splitting are:
\begin{align*}
\ext^3_8 \R^8 =& i(\R^8)\Omega_0, \\
\ext^3_{48} \R^8 =& \ker (\cdot \wedge \Omega_0 \colon \ext^3 \R^8 \to \ext^7 \R^8).
\end{align*}

Finally, a $\S7$ structure on the Riemannian manifold $(M,g)$ determines a canonical splitting of $\ext^k T^*M$. If we take the $\S7$ reduction $R$ of the $\SO(8)$ principal bundle given by the Cayley frames,  then those are given by $\ext_l^k T^*M = R \times_{\S7} \ext_l^k\R^8$.

\section{The intrinsic torsion} \label{tor}

We are going to compute the intrinsic torsion of the Levi-Civita connection, $\tor \in TM \otimes \ext_7^2T^*M$. Recall that the Levi-Civita connection $\nabla$ on $TM$ induces a connection $\omega$ on $\pso{M}$.  Then a connection on the $\S7$ reduction $R$ is defined by $\omega'=\proy (\omega)|_{TR}$, where $\proy$ denotes the orthogonal projection to $\7s$. The connection that $\omega'$ induces on $TM$ is denoted by $\con$ and determines the intrinsic torsion by means of the expression:
 \[\nabla_X Y= \con_X Y + \tor(X)Y.\]

The skew-symmetric endomorphism $\tor(X)$ can be identified with a $2$-form which lies in $ R\times_{\S7}\m = \ext_{7}^2 T^*M$ for each $X\in TM$. To compute it, we will first prove that the vector bundles $\ext_7^2T^*M$ and $H=\langle \eta \rangle ^\perp$ are isomorphic: \begin{lema} \label{torin}
There is a well defined $\S7$-equivariant map \[ \ext^2 T^*M \to H, \quad \alpha \longmapsto \alpha\eta, \]
whose kernel is  $\ext_{21}^2 T^*M$. Indeed, its restriction $c\colon \ext_7^2 T^*M \to H$ is an isomorphism whose inverse is given by $ (c^{-1}\phi) (X,Y)= \frac{1}{4}(\phi, (XY + g(X,Y))\eta).$
\end{lema}
\begin{proof}
The spinor $\beta \eta$ is perpendicular to $\eta$ if $\beta \in \ext^2T^*M$. Therefore, the map is well-defined and it is $\S7$-equivariant since $\S7=\Fix(\eta_p)$.

To prove that $c$ is an isomorphism, we first claim that if $(e_0,\dots,e_7)$ is a Cayley frame then $\alpha_j\eta = 4e^{0j}\eta$. Observe that we only need to check this formula for $j=1$  since $c$ is $\S7$-equivariant and $\GG2=\S7\cap \Fix(e_0)$ acts transitively on the $6$-sphere generated by $(e_1,\dots,e_7)$.
In this case, $\alpha_1= e^{01} + e^{23} - e^{45} - e^{67}$ and if $(i,j)\in \{(2,3),(5,4),(7,6)\}$ we have that $\Omega(e_0,e_1,e_i,e_j)=1$. The previous equality means that $e_0\eta = e_{1ij}\eta$, so that $e^{01}\eta = e^{ij}\eta$. 

Moreover, since $\{ e^{0i}\eta \}_{i=1}^{7}$ is an orthonormal basis of $H$ we have that \[c^{-1}(\phi)=\frac{1}{4} \sum_{i=1}^{7}{ (\phi, e^{0i}\eta)\alpha_i}.\]
If $X=e_0$, $Y=e_1$ are orthonormal vectors then $\alpha_j(e_0,e_1)= (e^{0j} - i(e_0)i(e_j)\Omega)(e_0,e_1) = \delta_{j1}$. Hence, $c^{-1}\phi(e_0,e_1)= \frac{1}{4}(\phi,e_0e_1\eta)$.

Finally, by dimensional reasons the Clifford product with $\eta$ must vanish on  $\ext_{21}^2 T^*M$.
\end{proof}

The previous result enables us to find a formula for the intrinsic torsion:
 
 \begin{prop} The intrinsic torsion is given by  $\tor(X)= 2c^{-1}\nabla_X\eta$.
 \end{prop}
 \begin{proof}
We also denote by $\nabla$ and $\con$ the induced connections on the spinorial bundle. According to \cite[p. 60]{FR00} we have that:
\[\nabla_X \phi = \con_X \phi + \frac{1}{2} \tor(X)\phi, \]
where $\tor(X)$ acts on $\phi$ as a $2$-form.
Since the holonomy of the connection $\con$ is contained in $\S7$ and $\Fix(\eta_p)=\S7 $ we have that $\con \eta=0$. Finally, if $X\in TM$ then $\nabla_X \eta \in H$ and $\tor(X) \in R\times_{\S7}\m$ thus, Lemma \ref{torin} shows that $\tor(X)= 2c^{-1}\nabla_X\eta$.
\end{proof}

\section{Spinorial classification of $\S7$ structures} \label{s}

Spin structures are classified \cite{MF86} according to the $\S7$ irreducible parts of $*d\Omega$ on $\ext^3T^*M$ in the following pure types:
%Recall that the representation of $\S7$ at $\ext^3 T^*M$ leads to an orthogonal descomposition into irreducible parts, which are: %$\wedge^2 T^*M=\wedge_7^2 T^*M \oplus \wedge_{21}^2 T^*M$ in
% the $\S7$ irreducible components of the representation on $\wedge^2T^*M$ are given by:

\begin{defi} \label{clasi}
A $\S7$-structure given by $\Omega$ is said to be:
\begin{enumerate}
\item[1.] Parallel, if $*d\Omega=0$.
\item[2.] Locally conformally parallel, if $*d\Omega \in \ext_8^3 T^*M $. 
\item[3.] Balanced if $*d\Omega \in \ext_{48}^3 T^*M$.
\end{enumerate} 
The Lee form of $\Omega$ is the unique $\theta \in \ext^1 T^*M$ such that the orthogonal projection to $\ext^5_8 T^*M$ of $d\Omega$ is $\theta\wedge \Omega$.
\end{defi} 
\begin{rem} Suppose that the structure is locally conformally parallel. Let $O$ be a contractible open set, take a primitive $f$ of $-\frac{1}{4}\theta|_O$ and define the metric $g'=e^{2f}g|_O$. The associated $\S7$ structure is $\Omega'=e^{4f}\Omega|_O$ and it verifies $d\Omega'=0$. Therefore, $\Omega|_O$ is conformal to a parallel structure. This justifies the name.
\end{rem}

 In order to rewrite this classification by means of $\eta$, we are going to calculate $*d\Omega$. For this purpose, consider the Dirac operator $D$ at $\spi(M)$ and the vector field $V$ such that \begin{align} \colorlabel{blue} \label{eqn3}  D\eta= V\eta.
\end{align} 
Then, the $3$-form $\gamma_8(X,Y,Z)= (D\eta, (X \times Y\times Z ) \eta)= (i(V)\Omega)(X,Y,Z)$ obviously lies in $\ext_8^3 T^*M$. 
\begin{prop}  \label{domega} 
Using the previous notation, we have:
\[*d \Omega = 2(\gamma_8 - 12\alt{c^{-1} \nabla \eta}) ,\]
where $\alt{\rt}  (X_1,\dots,X_n) = \frac{1}{n!} \sum_{\sigma \in S_n}(-1)^{\sgn{\sigma}} \rt (X_{\sigma(1)}, \dots, X_{\sigma(n)})$ if $\rt$ is a section of $\otimes^n TM$.

\end{prop}
\begin{proof}
Since $\nabla$ is a metric connection on the spinorial bundle and acts as a derivation for the Clifford product, we get:
\begin{align*}
(\nabla_T\Omega)(W,X,Y,Z) &= \frac{1}{2} \Big( ((-WXYZ + WZYX)\nabla_T\eta,\eta) + ((-WXYZ + WZYX)\eta,\nabla_T\eta)\Big) \\ & = \frac{1}{2} ((-ZYXW + XYZW - WXYZ + WZYX)\eta,\nabla_T\eta).
\end{align*}
Take orthonormal vectors $X,Y,Z$ and an orthonormal oriented basis $(X_0,\dots X_7)$ such that $X_0=X$, $X_1=Y$ and $X_2=Z$. Then,
\begin{align*}
\delta \Omega (X,Y,Z) & = - \sum_{i=3}^7 {\nabla_{X_i}\Omega(X_i,X,Y,Z)} 
= - 2\sum_{i=3}^7{(XYZ\eta,X_i\nabla_{X_i}\eta)} \\ 
&= - 2(D\eta,(X \times Y\times Z ) \eta)+ 2(XYZ\eta, X\nabla_X \eta + Y\nabla_Y \eta + Z\nabla_Z\eta)  \\
&= - 2((D\eta,(X \times Y\times Z )\eta) - (YZ\eta,\nabla_X\eta) + (XZ\eta, \nabla_Y\eta ) - (XY\eta, \nabla_Z\eta) \\
&= -2((D\eta, (X \times Y\times Z ) \eta) - 12\alt{ c^{-1}\nabla \eta} (X,Y,Z)).
\end{align*}
Note that the coefficient $12$ comes from the normalization of $\mathrm{alt}$ and the expression $c^{-1}(\nabla_X\eta)(X,Y) = \frac{1}{4}((XY+ g(X,Y))\eta,\nabla_X\eta)$.
\end{proof}

We are going to decompose $*d\Omega$ according to the previous splitting. 

\begin{lema} \label{gamma48} 
The $3$-form $ \gamma_{48}= 3\gamma_8-84\alt{ c^{-1}\nabla\eta}$ lies in $ \ext_{48}^3T^*M$ and \[*d\Omega=\frac{2}{7}\gamma_{48} + \frac{8}{7}\gamma_8.\]
\end{lema}
\begin{proof}
Take a unitary vector $X$ and a Cayley frame $(e_0,e_1,\dots, e_7)$ such that $X=e_0$. Then:
\begin{align*}
 (\gamma_8 \wedge \Omega) (e_1,\dots,e_7) =& (D\eta, (e_{123}-e_{145}-e_{167}-e_{246}+  e_{257}- e_{347} - e_{356})\eta)\\ =&7(D\eta,e_0\eta)= 7V^*(X), \\
(12\alt {c^{-1}\nabla\eta}\wedge \Omega) (e_1,\dots,e_7) =&  \cic  (\nabla_{e_1}\eta, e_{23}\eta) - \cic (\nabla_{e_1}\eta, e_{45}\eta) - \cic (\nabla_{e_1}\eta, e_{67}\eta) \\
& - \cic (\nabla_{e_2}\eta, e_{46}\eta)  + \cic (\nabla_{e_2}\eta, e_{57}\eta) - \cic (\nabla_{e_3}\eta, e_{47}\eta) \\ & - \cic (\nabla_{e_3}\eta, e_{56}\eta) =3(D\eta,e_0\eta)=3V^*(X).
\end{align*}

We denoted by $\cic$ the cyclic sums in the indices involved. To arrange the last term observe that each index appears $3$ times and:
\begin{align*} \cic (\nabla_{e_1}\eta, e_{23}\eta) =& (e_1\nabla_{e_1}\eta + e_2\nabla_{e_2} \eta + e_3\nabla_{e_3} \eta, e_{123}\eta) = (e_1\nabla_{e_1}\eta + e_2\nabla_{e_2}\eta + e_3\nabla_{e_3}\eta , e_{0}\eta), \\
 - \cic(\nabla_{e_1}\eta, e_{45}\eta) =& (e_1\nabla_{e_1}\eta + e_4\nabla_{e_4} \eta + e_5\nabla_{e_5} \eta, -e_{145}\eta) = (e_1\nabla_{e_1}\eta + e_4\nabla_{e_4}\eta + e_5\nabla_{e_5} \eta, e_{0}\eta),
\end{align*}
and so on. Note that we have used, as in the proof of Lemma \ref{torin}, that $e_{123}\eta = e_0\eta = - e_{145}\eta$. 

Since Cayley bases are positive oriented, we get $*(V^*)=\frac{1}{7}(\gamma_{8}\wedge \Omega)= 4\alt{c^{-1}\nabla \eta}$, so that $ \gamma_{48}$ as defined above lies in $\ext_{48}^3T^*M$. Finally, taking into account the formula for $*d\Omega$ in Proposition \ref{domega}, we get $*d\Omega=\frac{2}{7}\gamma_{48} + \frac{8}{7}\gamma_8$.
\end{proof}

We can now rewrite the classification of $\S7$ structures:
\begin{theo} \label{ppal} The $\S7$ structure given by $\Omega$ is: 
\begin{enumerate}
\item[1.] Parallel if $\nabla \eta=0$.
\item[2.] Locally conformally parallel if $i(V)\Omega= 28\alt {c^{-1}\nabla \eta}$.
\item[3.] Balanced if $D\eta=0$.
\end{enumerate}
Moreover, the Lee form is given by $\theta = \frac{8}{7}V^*$, where $V$ is defined as in the equation (\ref{eqn3}).
\end{theo}
\begin{proof}
It is an immediate consequence of Definition \ref{clasi} and Lemma \ref{gamma48}. To compute the Lee form we have used that the projection of $d\Omega$ to $\ext_{8}^5T^*M$ is $-\frac{8}{7}*\gamma_{8}$ and the formula $i(X)\Omega = *(X^*\wedge \Omega)$, which can be easily checked taking a Cayley frame and $X=e_0$.
\end{proof}

\section{The characteristic connection.} \label{car}

The characteristic connection of a $\S7$ structure is a connection $\car$ with totally skew-symmetric torsion, such that $\car\Omega = 0$. 
The computations above allow us to prove the existence and uniqueness of the characteristic connection for manifolds with a $\S7$ structure. This is a well known result which appears in \cite[Theorem 1.1]{IV04}. Our proof is based on the argument of Theorem 3.1 in \cite{FR03}. 

Consider the $\S7$-equivariant maps which are given in terms of a local Cayley frame:
\begin{align*}
& \Theta \colon \ext^3 T^*M \to TM\otimes \ext^2_7 T^*M, \quad \beta \longmapsto \Theta (\beta) =\sum_{j=0}^7{e_j\otimes \proy_{7}(i(e_j) \beta)},
\\
& \Xi \colon TM \otimes \ext^2_7 \to \ext^3 T^*M,\quad \alpha \otimes \beta \longmapsto \alpha\wedge \beta = 3\alt{\alpha\otimes\beta},
\end{align*}
where $\proy_7 \colon \ext^2 T^*M \to \ext_7^2T^*M$ is the orthogonal projection. 

Note that 
the map $\Xi \circ \Theta$ is symmetric and $\S7$-equivariant, so that its eigenspaces must be $\ext_8^3 T^*M$ and $\ext_{48}^3 T^*M$. Taking $i(e_0)\Omega \in \ext_8^3 T^*M$ and $e^{123} + e^{145}\in \ext_{48}^3 T^*M$ one can show that the eigenvalues are $\frac{9}{4}$ on $\ext_8^3 T^*M$ and $\frac{1}{2}$ on $\ext_{48}^3 T^*M$. 

\begin{prop} Given a $\S7$ structure, there exists a unique characteristic connection whose torsion $\rt \in \ext^3 T^*M$ is given by: \[ \rt= -\delta\Omega - \frac{7}{6}*(\theta \wedge \Omega). \]
\end{prop}
\begin{proof}
A connection with skew-symetric torsion $ \rt \in \ext^3 T^*M$ is given by $\nabla_X Y + \frac{1}{2}\rt(X,Y,\cdot)^\#,$ where $\rt(X,Y,\cdot)^\#$ is the vector field such that $(\rt(X,Y,\cdot)^\#)^*=\rt(X,Y,\cdot)$. Thus, the lift to the spinorial bundle is 
 $ \nabla_X \phi + \frac{1}{4}i(X)\rt \phi.$

Since the condition $\car \Omega=0$ is equivalent to $\car \eta =0$ and the kernel of the Clifford product by $\eta$ on $\ext^2T^*M$ is $\ext_{21}^2T^*M$, the set of characteristic connections is isomorphic to the set of $3$-forms $ \rt \in \ext^3T^*M$ such that
\[-4c^{-1}\nabla_X \eta= i(X)T\eta= \proy_{7}(i(X)T), \quad \forall X \in  \mathfrak{X}(M).\] 
The last equality may be rewritten as $-4\Theta(c^{-1}\nabla \eta)= \Theta(\rt)$. From the definition of $\gamma_{48}$ given in Lemma \ref{gamma48} we have: $-4 \Xi (c^{-1}\nabla \eta) = -12 \alt{ c^{-1} \nabla \eta }= \frac{1}{7}(\gamma_{48} -3\gamma_8).$
Finally, taking into account the eigenvalues of $\Xi \circ \Theta$, we get:
\[
T=  \frac{1}{7}(2\gamma_{48} -\frac{4}{3}\gamma_8)= *d\Omega - \frac{4}{3}\gamma_8 = -\delta\Omega - \frac{7}{6}*(\theta \wedge \Omega).
\]
To obtain the second equality we have used the formula for $d\Omega$ from Lemma \ref{gamma48}. To check the last one, note that $\gamma_8=i(V)\Omega= *(V^*\wedge \Omega)= \frac{7}{8}*\theta \wedge \Omega$.
\end{proof}

\section{$\GG2$ distributions} \label{g2}

In this section we define the notion of $\GG2$ distribution on a $\S7$ manifold in terms of spinors, and we study the torsion of the structure with respect to a suitable connection on the  distribution.
Then, we relate the $\S7$ structure of the ambient manifold with the $\GG2$ structure of the distribution. This approach enables us to study $\GG2$ structures on submanifolds of $\S7$ manifolds, $S^1$-principal fibre bundles over $\GG2$ manifolds and warped products of manifolds admitting a $\GG2$ structure with $\R$. Our analysis is very similar to the description of $\GG2$ structures from a spinorial viewpoint, done in \cite{ACFH15}, which we briefly recall.

A $7$-dimensional Riemannian manifold $(\sv,g)$ can be equipped with a $\GG2$ structure if it is spin and its spinorial bundle $\spi(\sv)$ admits a unitary section $\eta$. A cross product is then constructed from the spinor and is determined by a $3$-form $\Psi$. %In \cite{ACFH15} the integrability of the $\GG2$ structure is studied in spinorial terms, which we briefly recall.
Denote by $\nabla^Q$ both the Levi-Civita connection of the manifold and its lift to the spinorial bundle; an endomorphism $\sss$ of $T\sv$ is defined by the condition:
\[ \nabla_X^\sv \eta = \sss(X)\eta.\]
The intrinsic torsion is $-\frac{2}{3}i(\sss)\Psi$ \cite[Proposition 4.4]{ACFH15}, so that pure types of $\GG2$ structures are given by the $\GG2$ irreducible components of $\End(T\sv)$. It is known that $\End(\R^7)=\chi_1\oplus \chi_2 \oplus \chi_3 \oplus \chi_4$, where $\chi_i$ are irreducible $\GG2$ representations, defined by:
\[ \chi_1=\langle \Id \rangle , \quad \chi_2= \2g, \quad \chi_3= \Sym(\R^7), \quad \chi_4 =\{ A \colon \R^7  \to \R^7 \colon A(X) = X \times S, \quad S \in \R^7\},\]
where $\Sym(\R^7)$ denotes the set of symmetric and traceless endomorphisms. The dimensions of the previous spaces are $1$, $14$, $27$ and $7$ respectively.

If we denote by $R_{\sv}$ a $\GG2$ reduction of the $\SO(7)$ principal bundle $\pso{\sv}$  and define $\chi_i(\sv)= R_\sv\times_{\GG2} \chi_i$, then the pure classes of $\GG2$ structures are determined by the condition $\sss \in \chi_i(\sv)$. For instance, nearly parallel $\GG2$ structures verify $\sss \in \chi_1(\sv)$, almost parallel or calibrated are those with $\sss \in \chi_2(\sv)$, and locally conformally calibrated are such that $\sss \in \chi_4(\sv)$. Indeed in the nearly parallel case it holds that $\sss(X)=\lambda_0 X$ for some $\lambda_0\in \R$.  Moreover mixed classes are also relevant, for instance %integrable and 
cocalibrated structures which correspond to
% respectively to $\sss \in \chi_1(\sv)\oplus \chi_3(\sv) \oplus \chi_4(\sv)$ and 
 $\sss \in \chi_1(\sv)\oplus \chi_3(\sv)$. \\

Taking this into account, we define $\GG2$ structures on distributions and characterise  the existence of such structures.
\begin{defi}
Let $(M,g)$ be an oriented $8$-dimensional Riemannian manifold  and let $\ddd$ be a cooriented distribution of codimension $1$. We say that $\ddd$ has a $\GG2$ structure if the principal $\SO (7)$ bundle $\pso{\ddd}$ is spin and the spinorial bundle $\spi(\ddd)$ admits a unitary section.
\end{defi}

\begin{lema} \label{dco} Consider an oriented $8$-dimensional Riemannian manifold $(M,g)$ and a cooriented distribution $\ddd$ of codimension $1$. Take a  unitary vector field $N$, perpendicular to $\ddd$ such that $TM=\langle N \rangle \oplus \ddd$ as oriented bundles. 
The manifold $M$ is spin if and only if the bundle $\pso{\ddd}$ is spin. In this case, the spinorial bundles are related by $\spi(\ddd)=\spi^+(M)$ and it holds 
\begin{align*} \colorlabel{blue} \label{eqclif}
 X \cdot _\ddd \phi =NX \phi, \mbox{ if } X \in \ddd, \quad \phi \in \spi(\ddd), 
\end{align*}
where we have suppressed the symbol $\cdot_ M$ to denote the Clifford product on $M$.

Therefore, $M$ has a $\S7$ structure if and only if $\ddd$ has a $\GG2$ structure.
\end{lema}
\begin{proof}
The bundle $\pso{\ddd}$ is a reduction of $\pso{M}$, since we have the following inclusion:
\[i\colon \pso{\ddd} \to \pso{M}, \quad (X_1,\dots,X_7) \to (N,X_1,\dots,X_7).\] 

Suppose that $\pso{\ddd}$ is spin and denote the spin bundle by $\pi_\ddd \colon \pspin{\ddd} \to \pso{\ddd}$. Then, we can define the principal $\Spin8$ bundle  $\pspin{M}= \pspin{\ddd}\times_{\S7} \Spin8$ and the map:
\[ \pi_M \colon \pspin{M} \to \pso{M}, \quad [\tilde{F},\tilde{\varphi}] \to Ad(\tilde{\varphi})(i(\pi_\ddd(\tilde{F}))),\]
which is a double covering and $Ad$-equivariant. Therefore, $M$ is spin. Reciprocally, if $M$ is spin then the pullback $i^*(\pspin{M})$ is the spin bundle of $\pso{\ddd}$.

Moreover, the irreducible $8$-dimensional representation of $\Cl_7$ which maps the volume form to the identity can be constructed from the composition
\begin{align*} 
&\Cl_7 \to  \Cl_8^0 \xrightarrow{\rho} \Gl(\Delta^+), 
\end{align*}
where the first map is induced by the embedding $\R^7 \to \Cl_8^0$, $v\to e_0v$, denoting by $(e_0,\dots,e_7)$ the canonical basis of $\R^8$.

 Therefore, the spinorial bundle  $\spi(\ddd)$ coincides with $ \spi(M)^+$ and the Clifford products are related by the formula (\ref{eqclif}).

\end{proof}

From now on we assume that the manifold $(M,g)$ has a $\S7$ structure $\Omega$, constructed from a unitary section $\eta$ of the spinorial bundle $\spi(M)^+$, as in Definition \ref{spinform}. We equip $M$ with a distribution as in Lemma \ref{dco}. 
 \begin{obs}In this situation, we have the following:

\begin{enumerate} 
 
\item[1.] If $\beta \in \ext^{2k}T^*\ddd$ and $\phi \in \spi(\ddd)$ then $\beta \cdot_{\ddd}\phi= \beta \phi$.
\item[2.] There is an orthogonal decomposition $ \spi(\ddd)=\langle \eta \rangle \oplus (\ddd \cdot _\ddd \eta)$.
\item[3.] The section $\eta$ defines a cross product on $\ddd$ by means of:
 \[ (X\times Y)\eta = (X^*\wedge Y^*)\eta = (XY + g(X,Y)) \eta,\]
 which is determined by $\Psi_\ddd(X,Y,Z)=(X\eta,(Y\times Z)\eta)= -(\eta, XYZ\eta).$ \item[4.] The cross product is determined by $\Psi_\ddd = i(N)\Omega$. Therefore, using that $*\Omega=\Omega$ we get $\Omega = N^* \wedge \Psi_\ddd + *_\ddd \Psi_\ddd$.
\end{enumerate}

\end{obs}

We equip $\ddd$ with a suitable connection which is determined by the covariant derivative of the ambient manifold.

\begin{defi}
 The covariant derivative of $\ddd$ induced by $M$, $\nabla^\ddd$, is given by the expression:
 \[ \nabla^M _X Y = \nabla^\ddd _X Y + g(\ttt(X),Y)N,\quad  X,Y\in \ddd,\]
where $\ttt \in \End(\ddd)$ is given by: $2g(\ttt(X),Y) = -N(g(X,Y)) - g([X,N],Y) - g([Y,N],X) +  g([X,Y],N)$. 

\end{defi}

We will decompose $\ttt$ into its symmetric and skew-symmetric parts, which we call $\www$ and $\llll$ respectively.
The connection $\nabla^\ddd$ is a metric connection and the tensor $\llll = -\frac{1}{2}dN^*$ measures the lack of integrability of the distribution.

We will also denote by $\nabla^\ddd$ the lift of this connection to the spinorial bundle $\spi(\ddd)$. This connection is metric with respect to $(\cdot, \cdot)$ and behaves as a derivation with respect to the Clifford product. Hence, $\nabla^\ddd \eta \in \langle \eta \rangle ^\perp$, and there is an endomorphism of $\ddd$, that we will call $\sss_\ddd$, such that  $\nabla^\ddd_X \eta = \sss_\ddd(X)\cdot _\ddd \eta$. Therefore, if we define $\chi(\ddd)=R_\ddd \times \chi_i$, where  $R_\ddd$ is the  $\GG2$ reduction of $\pso{\ddd}$ determined by $\Psi_\ddd$, we have a splitting of $\End(\ddd)$ and we can decompose $\sss$ according to it:
 \[\sss_\ddd(X)=\lambda \Id + S_2 + S_3+ S_4, \]
where  $\lambda \in C^{\infty}(M)$,  $S_2\in \chi_2(\ddd)$, $S_3\in \chi_3(\ddd)$, $S_4 \in \chi_4(\ddd)$, and let $S\in \ddd$ be such that $S_4(X)=X\times S$. \\

We can relate these components with the $\S7$ structure defined on $M$. First of all, since $g(\nabla_X N, Y)=- g(\nabla_X Y,N)$ we get that the connection $\nabla^M$ at $\spi(M)^+$ in the direction of $\ddd$ is given by:
\[ \nabla_X^M \eta =  \nabla_X^\ddd \eta - \frac{1}{2} N\ttt(X)\eta = N \aaa(X)\eta, \]
where $\aaa = \sss_\ddd - \frac{1}{2} \ttt$.
We can decompose $\llll$ and $\www$ according to the splitting of $\End(\ddd)$ into irreducible parts and then decompose $\aaa$:
\begin{enumerate}
\item[1.] $\llll = L_2 + L_4$, where $L_2 \in \chi_2(\ddd)$, $L_4 \in \chi_4(\ddd)$ and let $L \in \ddd$ such that $L_4(X)=X \times L$.
\item[2.] $\www = h\Id + W_3$,  where $h\in C^{\infty}(M)$, $W_3\in \chi_3(\ddd)$. 
\item[3.] $\aaa = \mu \Id + A_2 + A_3 + A_4 $, where $\mu = \lambda - \frac{h}{2}$, $A_2 = S_2 - \frac{1}{2} L_2$, $A_3= S_3 -\frac{1}{2} W_3$, $A_4= S_4 -\frac{1}{2} L_4$. We will also denote $A= S - \frac{1}{2}L$.
\end{enumerate}

We are going to compute $*d\Omega$ in terms of the previous endomorphisms and $\nabla^\ddd_N\eta$.
Our first lemma is deduced from \cite[Theorems 4.6,4.8]{ACFH15}.
\begin{lema} \label{dirac}
If we take an oriented orthonormal local frame of $\ddd$, $(X_1,\dots,X_7)$ then  
\[\sum_{i=1}^7{X_i \aaa(X_i) \eta} = -7\mu\eta - 6NA\eta.\]
\end{lema}
\begin{proof}
We will split the endomorphism $\aaa$ into its $\GG2$ irreducible components and then compute each term separately. 
It is obvious that $\sum_{i=1}^7{X_i \mu X_i \eta}= -7\mu \eta$. Moreover, 
\[ \sum_{i=1}^7{X_i  (X_i \times A) \eta}=\sum_{i=1}^7{ X_i (X_i N A - g(X_i,A)N )\eta} = -6NA. \]
Finally consider the $\GG2$-equivariant map, $m\colon \ddd \otimes \ddd \to \spi(\ddd)$, $ m(X,Y) = XY\eta$. By dimensional reasons, its kernel must be $\chi_2(\ddd)\oplus\chi_3(\ddd)$. Therefore, if $k\in \{2,3\}$  we have that:
\[ \sum_{i=1}^7{X_i A_k(X_i) \eta} = m\left( \sum_{i=1}^7{(A_k)_{ij}X_iX_j}\right) =0,\]
where we have denoted $(A_k)_{ij}$ the entries of the matrix $A_k$ with respect to the basis $(X_1,\dots,X_7)$.
\end{proof}

\begin{obs} \label{RemU}  $\quad$ \\
\begin{enumerate}
\item[1.] Since $\nabla_N^M\eta$ is perpendicular to $\eta$ we can take $U\in \ddd$ such that $\nabla_N^M \eta = -NU\eta$.\\
In order to compute $\nabla_N^M\eta$ we may take $F=(X_0,X_1,\dots, X_7)$ a local orthonormal frame of $M$ such that $N=X_0$, a lifting $\tilde{F} \in \pspin{M}$ and write $\eta(p)= [\tilde{F}, s(p)]$.
With this notation we have:
\begin{align*} \colorlabel{blue} \label{eqnU}
 \nabla_{X_0}^M \eta =& [\tilde{F}, ds({X_0})]+ \frac{1}{2} \sum_{0\leq i<j \leq 7} g(\nabla_{X_0} X_i, X_j)X_iX_j \eta \\ =& [\tilde{F}, ds(X_0)]+ \frac{1}{2} \left( X_0 \nabla_{X_0} X_0 +  \sum_{1\leq i<j \leq 7} g(\nabla_{X_0} X_i, X_j)X_iX_j \right)\eta .
\end{align*}
Then, $U$ depends on the local information of the section and $\nabla_{X_0} X_i$.

\item[2.] The Dirac operator of $M$ is
 $$D^M\eta = U\eta + \sum_{i=1}^7{X_iN\aaa(X_i)\eta}= (U - 6A + 7\mu N)\eta.$$
 \end{enumerate}
\end{obs}

\begin{lema} \label{alt} If we define the forms on $\ddd$, $\beta_2\in \ext^2 \ddd^*$ and $\beta_3 \in \ext^3 \ddd^*$ by:
\[ \beta_2(X,Y)=g(A_2(X),Y), \quad \beta_3(X,Y,Z)=\alt { i(A_3)(\cdot)\Psi_\ddd}(X,Y,Z), \]
then:
\begin{enumerate}
\item[1.] $N^* \wedge i(N)(12 \alt{ c^{-1} \nabla \eta})= i(U-2A)(N^*\wedge \Psi_\ddd) -2N^*\wedge \beta_2,$
\item[2.] $12 \alt{ c^{-1} \nabla \eta}|_\sv = 3i (\mu N -A)\Omega|_\sv + 3\beta_3.$
\end{enumerate}
\end{lema}
\begin{proof}
The first equality is a consequence of the symmetric or skew-symmetric properties of each factor:
\begin{align*}
12 \alt{ c^{-1} \nabla \eta}(N,X,Y)=& -(XY\eta,NU\eta) - (NY\eta,N\aaa(X)\eta) + (NY\eta,N\aaa(X)\eta)\\
=& -i(U)\Psi_\ddd(X,Y) -2(Y\eta,(A_2(X) + X\times A)\eta)\\
=& \left( i(U-2A)(N^*\wedge \Psi_\ddd) -2N^*\wedge \beta_2 \right) (N,X,Y).
\end{align*}

To check the second one, note that $12\alt{ c^{-1} \nabla \eta }|_\sv= 3\alt {i(\aaa(\cdot))\Psi_\ddd}$. We compute separatedly each term in the decomposition of $\aaa$. 

It is evident that $3\alt {i(\mu \Id )\Psi_\ddd}(X,Y,Z)= 3\mu\Psi_\ddd(X,Y,Z)$ and $3\alt{i(A_3(\cdot))\Psi_\ddd}=3\beta_3$. Moreover, we have that $\alt{i(A_2(\cdot))\Psi_\ddd}=0$ because $A_2\in \chi_2(\sv)$. Finally, if $X$, $Y$ and $Z$ are orthonormal vectors in $TQ$, then:
\[i(A_4(X))\Psi_\ddd(Y,Z)=(X\times A \eta, Y \times Z \eta)= (XA\eta,YZ\eta)=-(A\eta,(X\times Y\times Z)\eta).\]
Therefore, $3\alt { i(A_4(\cdot))\Psi_\ddd}(X,Y,Z)= -3(A\eta,X\times Y \times Z \eta)$.
\end{proof}

From lemmas \ref{dirac} and  \ref{alt} and the decomposition of $*d\Omega$ obtained in Proposition \ref{gamma48}  we conclude:

\begin{prop} \label{distr} Let $U\in \ddd$ such that $\nabla_N^M \eta = -NU \eta$ and define the forms on $\ddd$, $\beta_2\in \ext^2 \ddd^*$ and $\beta_3 \in \ext^3 \ddd^*$ by:
\[ \beta_2(X,Y)=g(A_2(X),Y), \quad \beta_3(X,Y,Z)=\alt { i(A_3)(\cdot)\Psi_\ddd}(X,Y,Z). \]
Then, the pure components of $*d\Omega$ in terms of the $\GG2$ structure are:
\begin{align*}
(*d\Omega)_{48} &= \frac{2}{7}\left( -4i(A + U)N^*\wedge \Psi + 3i(A+U) *_\ddd \Psi_\ddd \right) + 4N^*\wedge \beta_2 - 6\beta_3,\\
(*d\Omega)_8 & = \frac{8}{7}i(U-6A +7\mu N) (N^* \wedge \Psi_\ddd + *_\ddd \Psi_\ddd) .
\end{align*}
\end{prop}

\subsection{Hypersurfaces}

Consider an $8$-dimensional $\S7$ manifold  $(M,g)$, whose $\S7$ form is constructed from a unitary section $\eta$ of the spinorial bundle $\spi(M)^+$, as in Definition  \ref{spinform}.
Let $\sv$ be an oriented hypersurface and take a unitary vector field $N$ such that $TM=\langle N \rangle \oplus T\sv$ as oriented vector bundles. 

The tubular neighbourhood theorem guarantees the existence of a cooriented distribution $\ddd$ defined on a neighbourhood $O$ of $\sv$ such that $\ddd|_\sv=TQ$. The coorientation is determined by a unitary extension of the normal vector field that we also denote by $N$. Both $\ddd$ and $\sv$ have $\GG2$ structures determined by the spinor $\eta$;   we are going to relate them using Proposition \ref{distr} in the manifold $O$. 

Note that the Levi-Civita connection of the hypersurface $\sv$ is $\nabla^\ddd|_\sv$. Moreover, $\llll|_{\sv}=0$ and $\www|_\sv$ is the Weingarten operator. Therefore, the restriction of $\sss_\ddd$ at $\sv$ is the endomorphism $\sss$ of the submanifold $\sv$. Decompose $\sss|_{\sv}$ and $\www|_{\sv}$ with respect of the $\GG2$ splitting of $\End(T\sv)$:
\begin{enumerate}
\item[1.] $\sss=\lambda \Id  + S_2 + S_3 + S_4$
\item[2.] $\www|_\sv = 7H\Id + W_3,$
\end{enumerate}
where  $\lambda \in C^{\infty}(M)$, $S_2\in \chi_2(\sv)$, $S_3,W_3\in \chi_3(\sv)$, $S_4\in \chi_4(\sv)$ and $H\in C^{\infty}(\sv)$ is the mean curvature. We will also denote by $S$ the vector in $T\sv$ such that $S_4(X)=X\times S$.

\begin{coro}  Let $U\in T\sv$ such that $\nabla_N^M \eta|_\sv = -NU \eta$ and $\Psi_\sv= i(N)\Omega$. Define the forms on $\sv$, $\beta_2\in \ext^2 T^*\sv$ and $\beta_3 \in \ext^3 T^*\sv$ by:
\[ \beta_2(X,Y)=g(S_2(X),Y), \quad \beta_3(X,Y,Z)=\alt { i((S_3 - \frac{1}{2}W_3)(\cdot))\Psi_\ddd }(X,Y,Z). \]
Then, the pure components of $*d\Omega$ in terms of the $\GG2$ structure are:
\begin{align*}
(*d\Omega)_{48} &= \frac{2}{7}\left ( -4i(S + U)N^*\wedge \Psi_\sv + 3i(S+U)*_\sv\Psi_\sv \right) + 4N^*\wedge i^*\beta_2 - 6\beta_3,\\
(*d\Omega)_8 & = \frac{8}{7}i\left (U-6S +7(\lambda - \frac{7}{2}H) N \right)(N^*\wedge \Psi_\sv +*_\sv\Psi_\sv ).
\end{align*}
\end{coro}

\begin{rem} \label{remU}
Note that the condition $\nabla_N\eta|_\sv=-NU\eta$ does not depend on the extension of the vectors.
Moreover, we can compute $U$ taking into account equation (\ref{eqnU}). The terms involved are extrinsic and not encoded in $\sss$ and $\WW$.
\end{rem}

Therefore, the $\S7$ type of the ambient manifold provides relations between the $\GG2$ type of the hypersurface, the vector $U$ and the Weingarten operator. Before stating the result, we recall that a hypersurface is said to be totally geodesic if $\www=0$, totally umbilic if $W_3=0$ and minimal if $H=0$.

\begin{theo} \label{G2type} Let $(M,g)$ be a Riemannian manifold endowed with a $\S7$ structure determined by a spinor $\eta$. Let $\sv$ be an oriented hypersurface with normal vector $N$ and let $U\in T\sv$ be such that $\nabla_N\eta|_\sv=-NU\eta$. 
\begin{enumerate}
\item[1.] If $M$ has a parallel $Spin(7)$ structure, then $\sv$ has a cocalibrated $\GG2$ structure. Moreover, 
 \begin{enumerate}
 	\item[1.1] $\sss=0$ if and only if $\sv$ is totally geodesic.
 	\item[1.2] $\sss \in  \chi_1(\sv)$ if and only if $\sv$ is totally umbilic.
 	\item[1.3] $\sss \in  \chi_3(\sv)$ if and only if $\sv$ is a minimal hypersurface. 
 \end{enumerate}
\item[2.] If $M$ has a locally conformally parallel $Spin(7)$ structure, then $\sss \in \chi_1(\sv)\oplus \oplus \chi_3(\sv) \oplus \chi_4(\sv)$. Indeed,
 \begin{enumerate}
 	\item[2.1] $\sss \in \chi_1(\sv)$ if and only if $U=0$ and $\sv$ is totally umbilic.
 	\item[2.2] $\sss \in  \chi_1(\sv) \oplus \chi_4(\sv) $ if and only if $\sv$ is totally umbilic. 

 	%\item[1.3] $S\in  \chi_1(\sv)\oplus \chi_2(\sv)\oplus \chi_3(\sv) \oplus \chi_4(\sv)$ if and only if $U\in TQ$. 
 \end{enumerate}
\item[3.] If $M$ has a balanced $Spin(7)$ structure, then:
 \begin{enumerate}
 	\item[3.1]  $\sss \in  \chi_2(\sv)\oplus \chi_3(\sv)$  if and only if $U=0$ and $\sv$ is a minimal hypersurface.
 	\item[3.2] $\sss \in  \chi_1(\sv)\oplus \chi_2(\sv)\oplus \chi_3(\sv)$ if and only if $U=0$.
 	\item[3.3] $\sss \in  \chi_2(\sv)\oplus \chi_3(\sv)\oplus \chi_4(\sv)$ if and only if $\sv$ is a minimal hypersurface.

 \end{enumerate}
 
\end{enumerate}
\end{theo}

\begin{proof}
The parallel case follows from the equalities $U=S=0$, $S_2=0$, $2\lambda=7H$ and $2S_3=W_3$. The locally conformally parallel case follows from the equalities $U=-S$, $S_2=0$ and $2S_3=W_3$,  which imply that $S\in  \chi_1(\sv)\oplus  \chi_2(\sv)\oplus  \chi_3(\sv)$. Finally the balanced case follows from $U=6S$ and $2\lambda=7H$.
\end{proof}

\subsection{ Principal bundles over a $\GG2$ manifold}
Let $\sv$ be a $\GG2$ manifold and let $\pi \colon M \to \sv$ be a $G= \R$ or $G=S^1$ principal bundle over $\sv$; identify its Lie algebra $\g$ with $\R$. 

Define the vertical field $N(p)= \frac{d}{dt}\Big|_{t=0}(p\exp(t))$. A connection $\omega \colon TM \to \g$  defines a horizontal distribution $\hhh$ and we can define a metric on $M$ such that:
\begin{enumerate}
 \item[1.] The map $d\pi \colon \hhh_p \to T_{\pi(p)}Q$ is an isometry
 \item[2.] The vector $N(p)$ is unitary and perpendicular to $\hhh_p$.
\end{enumerate}

The projection $d\pi$ induces a map $p\colon \pso{\hhh} \to \pso{\sv}$ so that the pullback to $\pspin{\sv}$ defines a spin structure $\pspin{\hhh}$ on $\pso{\hhh}$.  The map $\tilde{p} \colon \pspin{\hhh} \to \pspin{Q}$, which is canonically defined, has the property that $\tilde{p}(\tilde{\varphi}\tilde{F})=\tilde{\varphi} \tilde{p}(\tilde{F})$ if $\tilde{\varphi}\in \Spin8$, inducing therefore a map between the spinorial bundles, which we call $\bar{p}$.  
Note that this map gives isomorphisms $\spi(\hhh)_p \to \spi(\sv)_{\pi(p)}$. Moreover, let $X \in TQ$ and denote by $X^h$ its horizontal lift, then $\bar{p}(X^h\cdot _{\hhh} \phi)= X\bar{p}(\phi).$ Therefore, a section $\bar{\eta} \colon Q \to \spi(Q)$ allows us to define a section $\eta \colon M \to \spi(\hhh)$ by means of the expression $\bar{p}(\eta) = \bar{\eta}$. If we denote by $\Psi_\sv$ the $\GG2$ form on $\sv$, then $\Psi_\ddd= \pi^*\Psi_\sv$.

Furthermore, one can check that $\nabla_{X^h}^\hhh Y^h= (\nabla_X^Q Y)^h$. Hence, if we take $\sss \in \End(\sv)$ such that  $\nabla_X^Q\bar{\eta}= \sss(X) \eta$, we get that the endomorphism of the distribution $\sss_\ddd$ is the lifting of $\sss$, that is:
\[ \nabla_{X^h}^\hhh \eta= \sss(X)^h \eta. \]

Therefore the distribution $\hhh$ and the manifold $\sv$ have the same type of $\GG2$ structure. In order to classify the $\S7$ structure on $M$, denote the curvature of the connection $\omega$ by:
\[ \KK(X,Y)=[X^h,Y^h]- [X,Y]^h \in \langle N \rangle, \quad X,Y \in T\sv. \]
Since $\KK(X,Y)\in \langle N \rangle$ we  also denote by $\KK$ the $2$-form that we obtain contracting the tensor with the metric. As a skew-symmetric endomorphism, we can decompose $ \KK= \bar{L}_2 + \bar{L}_4$ where $\bar{L}_4(X)=X\times \bar{L}$ for some $\bar{L}$ in $TQ$. 

%We also take $\bar{L}\in T\sv$ such that the component of $\KK$ in $\chi_4(\sv)$ is $g(X\times \bar{L},Y)=i(L)i(N)\Omega(X,Y)$.

 \begin{coro} \label{cprin}
Suppose that $\nabla_X^\sv \bar{\eta} = \sss(X) \cdot _\sv \bar{\eta}$ with $\sss(X)=\lambda \Id + S_2 + S_3+ S_4$ where $\lambda \in  C^{\infty}(\sv)$, $S_2 \in  \chi_2(\sv)$,  $S_3 \in \chi_3(\sv)$, $S_4\in \chi_4(\sv)$ and  let $S\in T\sv$ be such that $S_4(X)=X\times S$.  Define $\beta_2\in \ext^2 T^*\sv$ and $\beta_3 \in \ext^3T^*\sv$ by:
\[ \beta_2(X,Y)=g\left(S_2(X)-\frac{1}{4} \bar{L}_2(X),Y \right), \quad \beta_3(X,Y,Z)=\alt { i(S_3(\cdot))\Psi_\sv}(X,Y,Z). \]
The pure components of $*d\Omega$ in terms of the $\GG2$ structure are: 
\begin{align*}
(*d\Omega)_{48} &= \frac{2}{7}\left( -4i(S^h + \frac{1}{2}\bar{L}^h)N^*\wedge \pi^*\Psi_\sv + 3i(S^h + \frac{1}{2}\bar{L}^h)\pi^* (*_\sv \Psi_\sv) \right) -4N^*\wedge \pi^*\beta_2 + 6\pi^*\beta_3,\\
(*d\Omega)_8 & = \frac{8}{7} i\left( \frac{15}{4}\bar{L}^h - 6S^h +7\lambda N \right) (N^* \wedge \pi^*\Psi_\sv + \pi^*(*_\sv \Psi_\sv)). 
\end{align*}
\end{coro}
\begin{proof}
The result follows immediately from Proposition \ref{distr} once we check that $\www=0$, $g(\llll(X),Y)=\frac{1}{2}\pi^*\KK(X,Y) $, and $U=\frac{3}{4}\bar{L}^h$.

First of all, since the connection $\omega$ is left-invariant we have that $[X^h,N]=0$ if $X\in TQ$. Thus, $\www=0$. Moreover, $\llll(X^h)(Y^h)=\frac{1}{2} \KK(X,Y)$.  Furthermore, let $F=(X_1,\dots, X_7)$ be a local frame of $\hhh$ which lifts some local frame of $T\sv$. Take a lift $\tilde{F} \in \pspin{\hhh}$ and write $\eta(p)= [\tilde{F}, s(p)]$.
We denote $X_0=N$ and compute $U$ using the formula (\ref{eqnU}).

By definition, if $\bar{\eta}(\pi(p))= [\tilde{p}(\tilde{F}(p)), \bar{s}(\pi(p))]$ then $s(p)=\bar{s}(\pi(p))$ so that $ds_p(N)=0$. Besides, according to Koszul formulas we have:
\begin{align*}
\nabla_N N &= 0, \\
g(\nabla_N X_i, X_j)& = -\frac{1}{2} g([X_i,X_j],N)= -\frac{1}{2} g(N,\KK(d\pi(X_i),d\pi(X_j))).
\end{align*} 

Therefore, if we define $\gamma_i(X,Y)=g(\bar{L}_i(X),Y)$, for $i\in\{2,4\}$,  then: \[\nabla_N \eta = -\frac{1}{4} \pi^* \KK \eta= -\frac{1}{4}\pi^* \gamma_4 \eta = -\frac{3}{4}N\bar{L}^h\eta, \]
where we have used that $\pi^* \gamma_2\eta=0$ because $\2g \subset \7s=\ext_{14}^2 \R^8$ and $\pi^*\gamma_4=-i(N)i(\bar{L}^h)\Omega$ so that $\pi^*\gamma_4\eta= 3N\bar{L}^h\eta$, as we noted in the proof of Lemma \ref{torin}.

\end{proof}

\subsection{Warped products}
We analyze $\S7$ structures on warped products of a $\GG2$ manifold with $\R$. Recall that a warped product of two Riemannian manifolds $(X_1,g_1)$ and $(X_2,g_2)$ is $(X_1\times X_2, g_1 + f_1g_2)$ where $f_1\colon X_1 \to \R$ is a smooth function. Therefore, we have to distinguish two cases.

\subsubsection{} 
Consider a $\GG2$ manifold $(\sv,g)$ and a smooth function $f\colon \R \to \R$. Define the Riemannian manifold $(M=\sv\times \R, e^{2f}g + dt^2)$.  This is the so-called spin cone.

The distribution $\ddd= T\sv$ obviously admits a $\GG2$ structure. The spinorial bundle is given by  $\spi(M)^+=\spi(T\sv\times \R) = \spi(\sv)\times \R$ and Clifford products are related by  $ (X \cdot _\sv \phi,t) = e^{-f} X \cdot _\ddd (\phi,t)= e^{-f} \dt X(\phi,t)$ if $X\in T\sv$. In the last expression, we have suppressed the symbol $\cdot$ to denote the Clifford product on $M$.

A unitary section $\eta$ is constructed from a section $\bar{\eta} \colon \sv \to \spi(\sv)$ by defining $\eta \colon M \to \spi(\ddd),$ $\eta(x,t)=(\bar{\eta}(x),t).$
If we denote by $\Psi_\sv$ the $\GG2$ form on $\sv$, then $\Psi_\ddd= e^{3f} \pi^*\Psi_\sv$ and $*_{\ddd}(\Psi_\ddd)= e^{4f}*_\sv(\Psi_\sv)$. In addition, since $\nabla_X^\ddd Y = \nabla_X ^\sv Y$ when $X,Y \in T \sv$, we have that $ \nabla_{X}^\ddd \eta = e^{-f} \sss(X) \cdot _{\ddd} \eta$, if $X \in T\sv$ and  $\nabla_X^\sv\bar{\eta}= \sss(X) \bar{\eta} $. That is, $\sss_\ddd= e^{-f}\sss$.

\begin{coro}Suppose that $\nabla_X^\sv \bar{\eta} = \sss(X) \cdot _\sv \bar{\eta}$ with $\sss(X)=\lambda \Id + S_2 + S_3+ S_4$ where $\lambda \in  C^{\infty}(\sv)$, $S_2 \in  \chi_2(\sv)$,  $S_3 \in \chi_3(\sv)$, $S_4\in \chi_4(Q)$. Let $S\in T\sv$ be such that $S_4(X)=X\times S$.  Denote by $\Psi_Q$ the $\GG2$-form on $\sv$ and define $\beta_2\in \ext^2 T^*\sv$ and $\beta_3 \in \ext^3T^*\sv$ by:
\[ \beta_2(X,Y)=g\left(S_2(X),Y\right), \quad \beta_3(X,Y,Z)=\alt { i(S_3(\cdot))\Psi_\sv}(X,Y,Z). \]
The pure components of $*d\Omega$ in terms of the $\GG2$ structure are:
\begin{align*}
(*d\Omega)_{48} &= \frac{2}{7} \left( -4e^{2f}i(S)dt \wedge \pi^*\Psi_\sv + 3e^{3f}i(S)\pi^* ( * _\sv \Psi_\sv ) \right)+ 4e^{f}dt^*\wedge \pi^*\beta_2 - 6e^{2f}\pi^*\beta_3,\\
(*d\Omega)_8 & = \frac{8}{7}i\left(-6e^{-f}S +7(\lambda e^{-f} +\frac{1}{2}f')\dt \right)( e^{3f} dt \wedge \pi^*\Psi_\sv + e^{4f}\pi^* (* _\sv \Psi_\sv)).
\end{align*}

\end{coro}
\begin{proof}
The result follows immediately from Proposition \ref{distr} once we check that $\www= -f' Id $,  $\llll = 0$ and $U=0$.

Since the distribution $\ddd$ is integrable, we have that $\llll = 0$. Take an orthonormal frame of $TQ$, $(X_1,\dots,X_7)$ and note that $\www(X_i,X_j)=-f'e^{2f}\delta_{ij}$ so that $\www=-f'$. Moreover, using the Koszul formulas we get:
 \[ \nabla_{\dt} \dt=0= \nabla_{\dt} (e^{-f}X_i). \]
Therefore, using formula (\ref{eqnU}) we conclude that $\nabla_{\dt} \eta =  0.$
\end{proof}

\subsubsection{} 
Consider a $\GG2$ manifold $(\sv,g)$ and a smooth function $f\colon \sv \to \R$. Define the Riemannian manifold $(M=\sv\times \R, g + e^{2f} dt^2)$. 

The distribution $\ddd= T\sv$ obviously admits a $\GG2$ structure. The spinorial bundle is given by  $\spi(M)^+=\spi(T\sv\times \R) = \spi(\sv)\times \R$ and the Clifford products are related by  $ (X \cdot _\sv \phi,t) =  X \cdot _\ddd (\phi,t)= e^{-f} \dt X(\phi,t)$ if $X\in T\sv$. We have suppressed again the symbol $\cdot$ to denote the Clifford product on $M$.

A unitary section $\eta$ is constructed from a section $\bar{\eta} \colon \sv \to \spi(\sv)$ by defining $\eta \colon M \to \spi(\ddd),$ $\eta(x,t)=(\bar{\eta}(x),t)$. If we denote by $\Psi_\sv$ the $\GG2$ form on $\sv$, then $\Psi_\ddd= \pi^*\Psi_\sv$ and $*_\ddd (\Psi_\ddd)= *_\sv(\Psi_\sv)$. In addition, since $\nabla_X^\ddd Y = \nabla_X ^\sv Y$ when $X,Y \in T \sv$ , if we take $\sss\in \End(T\sv)$ with $\nabla_X^\sv\bar{\eta}= \sss(X) \bar{\eta} $, then  $\sss_\ddd=\sss$.  

\begin{coro}Suppose that $\nabla_X^\sv \bar{\eta} = \sss(X) \cdot _\sv \bar{\eta}$ with $\sss(X)=\lambda \Id + S_2 + S_3+ S_4$ where $\lambda \in  C^{\infty}(\sv)$, $S_2 \in  \chi_2(\sv)$,  $S_3 \in \chi_3(\sv)$, $S_4\in \chi_4(Q)$. Let $S\in T\sv$ be such that $S_4(X)=X\times S$. Denote by $\Psi_Q$ the $\GG2$-form on $\sv$ and define $\beta_2\in \ext^2 T^*\sv$ and $\beta_3 \in \ext^3T^*\sv$ by:
\[ \beta_2(X,Y)=g\left(S_2(X),Y\right), \quad \beta_3(X,Y,Z)=\alt { i(S_3(\cdot))\Psi_\sv}(X,Y,Z). \]
The pure components of $*d\Omega$ in terms of the $\GG2$ structure are:
\begin{align*}
(*d\Omega)_{48} &= \frac{2}{7}\left( -4i\left(S+ \frac{1}{2}\grad(f)\right)e^{f}dt \wedge \pi^*\Psi_\sv + 3i\left(S+ \grad(f)\right) \pi^*( *_\sv \Psi_\sv) \right) + 4e^{f}dt\wedge \pi^*\beta_2 - 6\pi^* \beta_3,\\
(*d\Omega)_8 & = \frac{8}{7}i\left( \frac{1}{2}\grad(f) -6S +7\lambda e^{-f}\dt \right)( e^{f}dt \wedge \pi^*\Psi_\sv + \pi^*( * _\sv \Psi_\sv)).
\end{align*}
\end{coro}

\begin{proof}
The result follows immediatly from Proposition \ref{distr} once we check that $\www= 0 $,  $\llll = 0$ and $U=\frac{1}{2} \grad(f)$.

Since the distribution $\ddd$ is integrable, we have that $\llll = 0$. Take an orthonormal frame of $TQ$, $(X_1,\dots,X_7)$ and note that $\www(X_i,X_j)=0$. Moreover, using the Koszul formulas we get:
\begin{align*}
g(\nabla_{e^{-f}\dt} X_i,X_j) & =0, \\
g\left(\nabla_{e^{-f}\dt} e^{-f}\dt, X_i \right) &= -X_i(f).
\end{align*}
Therefore, using formula (\ref{eqnU}) we conclude that  $\nabla_N \eta =  -\frac{1}{2}e^{-f} \left( \dt \right) \grad(f) \eta. $
\end{proof}

\section{$\S7$ structures on quasi abelian lie algebras} \label{sqab}
 
As an application of the previous section, we are going to study $\S7$ structures on quasi abelian Lie algebras. The geometric setting will be that of a simply connected Lie group with an invariant $\S7$ structure, endowed with an integrable distribution which inherits a $G_2$ structure. The integral submanifolds of the distribution are actually flat, so that the $G_2$ distribution will be parallel, but they will have non-trivial Weingarten operators. In some cases, finding a lattice in the Lie group will allow us to give compact examples.

First of all, let us recall the following definition:

\begin{defi} A Lie algebra $\g$ is called quasi abelian if it contains a codimension $1$ abelian ideal $\h$.
\end{defi}

The information of $\g$ is then encoded in $ad(x)$ for any vector $x$ transversal to $\h$. The following result shows that $\h$ is unique in $\g$ with exception of the Lie algebras $\R^n$ and $L_3\oplus \R^{n-3}$, where $L_3$ is the Lie algebra of the $3$-dimensional Heisenberg group, which is generated by $x$,$y$,$z$ with relations $[x,y]=z$ and $[x,z]=[y,z]=0$.

\begin{lema}\label{unique}  Let $\g$ be a $n$-dimensional quasi abelian Lie algebra with $n\geq 3$ .
If $\g$ is not isomorphic to $\R^n$ or $L_3\oplus \R^{n-3}$, then it has a unique codimension $1$ abelian ideal. 
Moreover, codimension $1$ abelian ideals on $L_3\oplus \R^{n-3}$ are parametrized by $\R \PP ^1$.
\end{lema}
\begin{proof}
Suppose that $\g$ is not isomorphic to $\R^n$ and let $\h$ be a codimension $1$ abelian ideal with a transversal vector $x$.
Let $\h'$ be a codimension $1$ abelian ideal different from $\h$. If $u \in \h$ is such that $x + u\in \h'$ and $v \in \h \cap \h'$, then  $0=[x + u , v]= ad(x)(v)$.
Since $\h\cap \h'$ is $(n-2)$-dimensional and $\g$ is not abelian we conclude that $\h \cap \h'=\ker(ad(x)|_\h)$ and $ad(x)(\h)=\langle z \rangle$ for some $z\in \h$. Take $y\in \h$ with $[x,y]=z$ and observe that $z\in [\g,\g]\subset \h'$, that is, $z\in \h \cap \h'$ and $[x,z]=0$. Therefore, $\g$ is isomorphic to $L_3 \oplus \R^{n-3}$.

Moreover, from the discussion above we get that $\h'=\langle v,z\rangle \oplus \R^5$ for some $v\in \langle x,y \rangle$. Conversely, all the subspaces of the previous form are actually codimension $1$ abelian ideals. Therefore, they are parametrized by $\R \PP^1$. 

%Take a metric $g$ on $L_3\oplus \R^{n-3}$ such that $x$ and $y$ are unitary, $x$ is perpendicular to $\h$ and $ad(x)(v)=y^*(v)z$.
%Let $w$ be a unitary vector, perpendicular to $\h'$ and consider $w'\in \langle x, y \rangle^\perp= \h \cap \h'$ such that $w=\lambda x + \mu y + w'$.
%Observe that $(\lambda,\mu)\neq 0$ because $[x,y]=z$. Moreover, since $g(w',w')(\lambda x + \mu y) - (\lambda^2 + \mu^2)w', \, \mu x - \lambda y \in \h$ we have that:
%\[0=[g(w',w')(\lambda x + \mu y) - (\lambda^2 + \mu^2)w', \mu x - \lambda y ]=- (\lambda^2 + \mu^2)g(w',w')z, \]
%therefore, $w'=0$. 
%Note also that all the hyperplanes in $\g$ which are perpendicular to a vector in $\langle x, y \rangle$ are actually codimension $1$ abelian ideals. Therefore, those are parametrized by $\R \PP^1$.

\end{proof}

An invariant $\S7$ structure on a Lie group is determined by the choice of a $\S7$ form $\Omega$, which is in turn determined by a direction of the spinorial space $\Delta^+$.  

Define the set $\MaxAb$  with elements $(\g, \h, g, \nu_g ,\Omega)$ where $\g$ is a non-trivial quasi abelian Lie algebra with a marked codimension $1$ abelian ideal $\h$, $g$ is a metric on $\g$, $\nu_g$ is a volume form on $\g$ and $\Omega$ is a $\S7$ structure on $(\g,g, \nu_g)$. We will say that $\varphi' \colon (\g, \h, g, \nu_g, \Omega) \to (\g ',\h', g ', \nu_{g'}, \Omega')$ is an isomorphism if $\varphi$ is an isomorphism of Lie algebras such that $\varphi'(\h)=\h'$, $(\varphi')^*g'=g$, $\varphi^*\nu_{g'}=\nu_g$ and $\varphi^{*}\Omega'=\Omega$.

\begin{lema} \label{lemIma}
The set $\IMA$ of isomorphisms classes of $\MaxAb$ is given by:
\[ \IMA = \left( (\End(\R^7)-\{0\} )\times \PP (\Delta^+) \right) / \OO (7) , \]
 where $\OO (7)$ acts via 
\begin{align} \colorlabel{blue} \label{eqno}
\varphi \cdot (\eee,[\eta]) = (\det(\varphi) \varphi\circ \eee  \circ \varphi^{-1}, [ \rho(\tilde{\varphi})\eta] ),
 \end{align}
where $\tilde{\varphi}$ is a lifting to $\Spin8$ of the unique $\varphi' \in \SO(8)$ such that $\varphi'|_{\R^7}=\varphi$.
\end{lema}
\begin{proof}
A map $(\End(\R^7)-\{0\} )\times \PP (\Delta^+)  \to \MaxAb$ can be defined as follows. Take $(\eee, \bar{\eta})$ and define the Lie structure on $\R^8$ with oriented basis $(e_0,\dots, e_7)$ such that $\R ^7= \langle e_1,\dots, e_7 \rangle$ is a maximal abelian ideal and $\eee=ad(e_0)|_{\R^7}$. We will endow this algebra with the canonical metric, the standard volume form and the spin structure determined by $\eta$.

It is obvious that a representative of each element of $\IMA$ can be chosen to live in the image of our map. Moreover, if two structures given by $(\eee,\bar{\eta})$ and $(\eee', \bar{\eta}')$ 
are isomorphic via $\varphi'$, we have the following:

\begin{enumerate}
\item[1.] $\varphi'(e_0)=\pm e_0$ and $\varphi = \varphi'|_{\R^7} \in \OO (7)$, since $\varphi'$ preserves the metric and the orientation. 
\item[2.] Denote by $\tilde{\varphi}$ any lifting of $\varphi'$ to $\Spin8$. Since $(\varphi')^*\Omega'=\Omega$, we have that $\Fix(\Omega)=(\varphi')^{-1}\circ \Fix(\Omega') \circ (\varphi')$, thus $\Fix(\eta)= \tilde{\varphi}^{-1}\Fix(\eta')\tilde{\varphi}$. But $\Fix(\rho(\tilde{\varphi})^{-1}\eta')=\tilde{\varphi}^{-1}\Fix(\eta')\tilde{\varphi}$, so that $\eta=\pm \rho(\tilde{\varphi})^{-1}\eta'$.
\item[3.] $\varphi \circ  \eee = det(\varphi) \eee ' \circ \varphi$, since $\varphi'$ is an isomorphism of Lie algebras.

\end{enumerate}
\end{proof}

From now on we denote by $(\R^8,\eee,[\eta])$ to $(\g,\h, g,\nu,\Omega)\in \MaxAb$ where $\g$ is the Lie algebra $\R^8$ with maximal abelian ideal $\h=\R^7$, $ad(e_0)=\eee$, $g$ is the canonical metric, $\nu$ is the canonical volume form and the $\S7$ form $\Omega$ is determined by $[\eta]$.

\begin{rem} To obtain an analogue of Lemma \ref{lemIma}, suppressing the condition $\varphi'(\h)=\h'$ in the definition of isomorphism, we have to treat separatedly the case of the Lie algebra $L_3\oplus \R^{5}$. For this purpose, define $\eee(x)=e_1^*(x)e_2$ and observe that lemmas \ref{unique} and \ref{lemIma} allow us to suppose that any isomorphism of structures with underlying Lie algebra $L_3\oplus \R^5$ is represented by  $\varphi' \colon (\R^8,\lambda \eee,[\eta]) \to (\R^8,\lambda' \eee, [\eta'])$, for some $\lambda,\lambda'\neq 0$. 

The set $\varphi'(\R^7)$ is a codimension $1$ abelian ideal, hence Lemma \ref{unique} guarantees that $\varphi'(e_0)= \cos(\theta)e_0 + \sin(\theta)e_1$.  Denote $\R^6=\langle e_2,\dots, e_7 \rangle$ and let $v,v'\in \R^6$ be such that $\varphi'(v)= -\mu \sin(\theta)e_0 + \mu \cos(\theta)e_1 + v'$. Then, $ 0=\varphi'[e_0,v]= [\cos(\theta)e_0 + \sin(\theta)e_1,-\mu \sin(\theta)e_0 + \mu \cos(\theta)e_1 + v' ]= \mu \lambda' e_2.$ Therefore $\mu=0$, $\R^6$ is $\varphi'$-invariant and $\varphi'(e_1)=\mp \sin(\theta)e_0 \pm \cos(\theta)e_1$. 

Denote by $\varphi_1$ the restriction of $\varphi'$ to $\langle e_0, e_1 \rangle$ and note that:
 $ \lambda \varphi'(e_2)= \varphi'[e_0,e_1]=[\varphi'(e_0),\varphi'(e_1)]= \det(\varphi_1) \lambda' e_2. $
Hence  $\varphi'(e_2)= \det(\varphi_1) \frac{\lambda'}{\lambda}e_2$ and $|\lambda|=|\lambda'|$. Then, $\varphi'$ is determined by $\varphi_1$ and $\varphi_2=\varphi'|_{\R^5}$, where $\R^5=\langle e_3,\dots, e_7\rangle$,  under the conditions $\frac{\lambda'}{\lambda}\det(\varphi_2)=1$ and $\varphi'(e_2)= \det(\varphi_1) \frac{\lambda'}{\lambda}e_2$.

The condition over the spinor is obviously $\eta'=\pm \rho(\tilde{\varphi})\eta$, where $\tilde{\varphi}$ is any lifting of $\varphi'$ to $\Spin8$.
\end{rem}

In the following result we describe the action which appears in Lemma \ref{lemIma}.

\begin{lema} \label{orbits} Under the action of $\OO (7)$ on $\End(\R^7)$,
\begin{align} \colorlabel{blue} \label{eqnp}
\varphi \cdot \eee=  det(\varphi) \varphi\circ \eee  \circ \varphi^{-1}, 
\end{align}
the sets $\langle \Id \rangle$, $\Sym(\R^7)$ and $\ext^2 \R^7$ are parametrized respectively by:
\begin{enumerate}
\item[1.] $[0,\infty)$,
\item[2.]  $\{ (\lambda_1,\dots, \lambda_7) \colon \lambda_i \leq \lambda_{j+1},  \sum_{j=1}^7{\lambda_i}=0 \}/ \sim$, where $(\lambda_1,\dots, \lambda_7) \sim (-\lambda_7, \dots, - \lambda_1)$,
\item[3.] $\{(\lambda_1,\lambda_2,\lambda_3) \colon 0\leq \lambda_1 \leq \lambda_2 \leq \lambda_3\} $.
\end{enumerate}
\end{lema}
\begin{proof}
The first claim is obvious and the second follows from the fact that each symmetric matrix has an oriented orthonormal basis of ordered eigenvectors. Note also that $-\Id \cdot \diag(\lambda_1,\dots,\lambda_7)= \diag(-\lambda_7, \dots, - \lambda_1)$, hence $(\lambda_1,\dots, \lambda_7)$ is related to  $(-\lambda_7, \dots, - \lambda_1)$.

If $\eee$ is a skew-symmetric endomorphism of $\R^7$ we can find a hermitian basis in $\C^7$ of eigenvectors and the eigenvalues are of the form $(-\lambda_3 i,-\lambda_2 i,\lambda_1 i,0,\lambda_1 i,\lambda_2 i,\lambda_3 i)$ with $0\leq \lambda_j\leq \lambda_{j+1}$. Moreover, the real parts of the eigenspaces associated to $-\lambda_j i$ and $\lambda_{j}i$ coincide. Thus, we can find a positive oriented orthonormal basis $(v_1,w_1,\dots,v_k,w_k,u_1,\dots u_{7-2k})$ of $\R^7$, such that $\eee(v_j)= \lambda_j w_j$ and $\eee(u_j)=0$.
Finally note that $(\lambda_1,\lambda_2,\lambda_3)$ are invariantly defined in the orbit.
\end{proof}

In Lemma \ref{lemIma}, the second factor of the product of $\MaxAb$ depends on $\Fix(\eee)$ under the action defined by (\ref{eqnp}) and it is determined by the number of equal eigenvalues.  Now we compute the invariants that we  defined for $\GG2$ distributions on $\R^7$:

\begin{prop} \label{clasiqab} Consider $(\R^8,\eee,[\eta])\in \MaxAb$ and decompose $\eee$ according to the $\GG2$ structure induced by $\eta$, that is  $ \eee = h \Id + E_2 + E_3 + E_4$, where $h\in \R$, $E_2 \in \chi_2$, $E_3\in \chi_3$, $E_4 \in \chi_4$ and $E_4(X)=X\times E$ for some $E \in \R^7$. Define $\Psi, \beta_3 \in \ext^3T^*\R^7$ by $\Psi = \Omega|_{\R^7 }$ and $\beta_3(X,Y,Z)=\alt { i(E_3(\cdot))\Psi}]$. We have:
\begin{align*}
(*d\Omega)_{48} &= \frac{2}{7}\left( 6i(E)e^0 \wedge \Psi - \frac{9}{4} i(E) *_{\R^7} \Psi \right) + 6\beta_3,\\
(*d\Omega)_8 & = - \left(  \frac{12}{7} E +  4 h e_0 \right) (e^0  \wedge \Psi +  *_{\R^7} \Psi) .
\end{align*}
\end{prop}

\begin{proof}
The result follows immediately from Proposition \ref{distr} once we check that: $\mu = -\frac{1}{2}h$, $A_2=0$, $ A_3= -\frac{1}{2} E_3$, $A=0$ and $U= - \frac{3}{2}E$.
 
To obtain this, first observe that $\nabla^\h \eta=0$ and $\llll = 0$ because $\h$ is an abelian ideal. From the formula of the Weingarten operator we get: $\www = h \Id + E_3$. 
To compute $U$ we use again equation (\ref{eqnU}), obtaining that:
 \[ \nabla_{N} \eta =\frac{3}{2}e_0 E\eta,\]
 since $\nabla_{e_0}e_0 = 0$ because $\h$ is an ideal and $\nabla_{e_0}{e_j}= (E_2 + E_4)(e_j)$ if $j>0$.

\end{proof}

In the next result we characterise in terms of Lemma \ref{orbits} the type of $\S7$ structure on quasi abelian Lie algebras. For this purpose, recall that a Lie algebra is called unimodular if the volume form is not exact. In the case of the Lie algebra $(\R^8,\eee)$, it is equivalent to say that $\eee$ is traceless. 

\begin{theo} \label{qab} Consider the Lie algebra $(\R^8,\eee)$ endowed with the standard metric and volume form. Denote by  $\eee_{13}$ and $\eee_{24}$ the symmetric and skew-symmetric parts of the endomorphism $\eee\neq 0$. Then, the Lie algebra admits a $\S7$ structure of type:
\begin{enumerate}
\item[1.] parallel, if and only if $\eee_{13}=0$ and $\eee_{24}$ is associated to $(\lambda_1,\lambda_2, \lambda_{1}+\lambda_{2})$ with $0\leq \lambda_1 \leq \lambda_2$, $\lambda_2>0$ as in Lemma \ref{orbits}.
\item[2.] locally conformally parallel and non-parallel if and only if $\eee_{13}=h\Id$ with $h\neq 0$ and $\eee_{24}$ is associated to $(\lambda_1,\lambda_2, \lambda_{1}+\lambda_{2})$ with $0\leq \lambda_1 \leq \lambda_2$, as in Lemma \ref{orbits}.
\item[3.] balanced if and only if it is unimodular and $\eee_{24}$ is associated to $(\lambda_1,\lambda_2, \lambda_{1}+\lambda_{2})$ with $0\leq \lambda_1 \leq \lambda_2$, as in Lemma \ref{orbits}.
\end{enumerate}
Moreover, if $\eee_{24}\neq 0$ then it admits a $\S7$ structure of mixed type.
\end{theo}
\begin{proof}
We identify $\eee_{24}$ with a $2$-form $\gamma$ which can be written with respect to a positive oriented orthonormal basis $(X_1,\dots,X_7)$ of $\R^7$ as $\gamma= \lambda_1X^{23} + \lambda_2 X^{45} + \lambda_3 X^{67}$,
where $0\leq \lambda_j\leq \lambda_{j+1}$ and $X^{ij}=X_i^*\wedge X_j^*$. 

Due to Proposition \ref{clasiqab}, to prove the first part we have to check that under the assumption $\eee_{24}\neq 0$, the existence of a spinor $\eta$ such that $\gamma\eta=0$ is equivalent to the fact that $\eee_{24}$ is associated to $(\lambda_1,\lambda_2, \lambda_{1}+\lambda_{2})$ with $0\leq \lambda_1 \leq \lambda_2$. This spinor exists if and only if $\rho_7(\lambda_1X_2X_3 + \lambda_2 X_4X_5 + \lambda_3 X_6 X_7)$ is non-invertible for some $8$-dimensional real irreducible representation $\rho_7  \colon \Cl_7 \to \End( \R^8)$ which maps the volume form $\nu_7$ to the identity, since they are all equivalent \cite[Proposition 5.9]{LM89}.

It is known that the two distinct irreducible representations of $\Cl_7$ can be constructed from the octonions $\mathbb{O}$ \cite[p. 51]{LM89}. Specifically, those are the extension to $\Cl_7$ of the maps $\rho_\theta \colon \R^7 \to \End(\R^8)$, $\rho_\theta(v)(x)=\theta vx$, where $\theta=\pm 1$ and $\R^7$ is viewed as the imaginary part of the octonions.  Define the isometry $\varphi$ of $\R^7$ which maps $X_i$ to $e_i$ and note that the volume form is fixed by the extension of $\varphi$ to the Clifford algebra. The extensions of $\rho_\theta$ and $\varphi$ to $\Cl_7$ are denoted in the same way. We check the previous condition using the representation $\rho_7=  \rho_\theta \circ \varphi \colon \Cl_7 \to \End(\R^8)$, taking $\theta$ such that $\rho_\theta(\nu_7)=\Id$. The determinant of $\rho_7(\lambda_1X_2X_3 + \lambda_2 X_4X_5 + \lambda_3 X_6 X_7)$ is given by:
\[ (\lambda_1 + \lambda_2 + \lambda_3)^2(\lambda_1 + \lambda_2 - \lambda_3)^2(\lambda_1 - \lambda_2 - \lambda_3)^2(\lambda_1 - \lambda_2 + \lambda_3)^2.\]
Since $\lambda_1\leq \lambda_2 \leq \lambda_3$, the endomorphism is non-invertible if and only if $\lambda_3=\lambda_2 + \lambda_1$.

Finally, if $\eee_{24}\neq 0$ then $\rho_7(\lambda_1X_2X_3 + \lambda_2 X_4X_5 + \lambda_3 X_6 X_7) \neq 0$ so that there is a spinor inducing a $\S7$ structure of mixed type. 
\end{proof}

Recall that solvmanifolds are compact quotients $G/\Gamma$, where $G$ is a simply connected solvable Lie group and $\Gamma$ is a discrete lattice. This forces the Lie algebra $\g$ of $G$ to be unimodular \cite[Lemma 6.2]{Mi}. Thefore, using 
Proposition  \ref{clasiqab}, we conclude the following:
\begin{coro} \label{solv} There exists no quasi abelian solvmanifold with an invariant locally conformally parallel and non-parallel $\S7$ structure.
\end{coro}

Of course, a torus is solvmanifold which admits a parallel $\S7$ structure. 

\begin{coro} \label{flat}
If $(\R^8,\eee)$ is a quasi abelian Lie algebra such that $\eee$ is skew-symmetric, then it is flat.
In particular, quasi abelian Lie algebras which admit an invariant parallel $\S7$ structure are flat.
\end{coro}
\begin{proof}
Let $(\R^8,\eee)$ be a quasi abelian Lie algebra and denote by $\eee_{13}$ and $\eee_{24}$ the symmetric and skew-symmetric parts of $\eee$. It is straightforward to check that if $i,j>0$ then:
\[ \nabla_{e_0}e_0=0, \quad \nabla_{e_0}e_j= \eee_{24}(e_j), \quad \nabla_{e_i}e_0= - \eee_{13}(e_i), \quad \nabla_{e_i}e_j= g(\eee_{13}(e_i),e_j)e_0. \]
From this, one can deduce that if $i,j,k>0$, then the curvature tensor is given by:
\begin{align*}
R(e_0,e_j)e_0 =& - (\eee_{24}\circ \eee_{13} + \eee_{13} \circ \eee_{24})(e_j), \\
R(e_0,e_j)e_k =& -g(\eee_{13}(e_k),(\eee + \eee_{24})(e_j))e_0, \\
R(e_i,e_j)e_0 =&\, 0, \\
R(e_i,e_j)e_k =&\, g(\eee_{13}(e_j),e_k)\eee_{13}(e_i) - g(\eee_{13}e_i,e_k)\eee_{13}(e_j). 
\end{align*}
Therefore, if $\eee$ is skew-symmetric then the Lie group is flat.

\end{proof}

\subsection*{Examples} \label{ex}
Let $\g$ be a quasi abelian Lie algebra determined by an endomorphism $\eee$. Consider the unique simply connected Lie group $G$ whose Lie algebra is $\g$.
The split exact sequence of Lie algebras $0 \rightarrow \h \rightarrow \g \rightarrow \g/ \h \rightarrow 0$  lifts to a split exact sequence of Lie groups $0 \rightarrow (\R^7,+) \rightarrow G \rightarrow (G/ \R^7 = \R,+) \rightarrow 0$. This splitting and the conjugation $\epsilon$ on $G$ by the elements of $(\R,+)$, provide an isomorphism  $(\R ,+)\ltimes_{\epsilon} (\R^7,+)$.  Therefore  $\frac{d}{d t}\big|_{t=s} d(\epsilon (t))= s\eee$, so that $d(\epsilon(t))=\exp(t\eee)=\epsilon (t)$, using that the exponential of $\R^7$ is the identity.
\subsection*{A nilmanifold with a balanced and a mixed $\S7$ structure.} 
Define the endomorphism of $\R^7$
\[\eee= \left( {\begin{array}{ccccccc}
               0 & -1 & 0 & 0 & 0& 0& 0 \\
               0 & 0 & -2 & 0 & 0& 0& 0  \\
               0 & 0&  0 & 0 & 0 & 0 & 0 \\
               0 & 0&  0 & 0 & -1 & 0 & 0 \\
               0 & 0 & 0 & 0 & 0& -1& 0 \\
               0 & 0  & 0 & 0 & 0 &0 & -1 \\
               0 & 0& 0 & 0 & 0& 0& 0 \\ \end{array} } \right), \]
and consider the quasi abelian Lie algebra $(\R^8, \eee)$. Note that this is a nilpotent Lie algebra with structure equations $(0,02,03,04,05,06,07,0)$, using Salamon notation \cite{S}. %with respect to Chevalley-Eilenberg differential $d\beta(X,Y)=-\beta([X,Y])$.

The symmetric part of $\eee$ is traceless and the eigenvalues of its skew-symmetric part are of the form  $(\lambda_1,\lambda_2, \lambda_1+\lambda_2)$. Therefore, Theorem \ref{qab} guarantees the existence of an invariant $\S7$ structure of type balanced and other invariant $\S7$ structure  which is mixed. To avoid computing the eigenvalues, one can observe that it  we take the standard form $\Omega_0$ in $\R^8$, determined by a spinor $\eta$, it holds that $e_2e_3\eta= -e_4e_5\eta= -e_6e_7\eta$ and $e_1e_2\eta = -e_5e_6\eta$. Therefore, if we identify the skew-symmetric part of $\eee$ with a $2$-form, $\gamma$, we get that $\gamma \eta=0$.

On some nilpotent Lie algebras, the existence of a lattice is guaranteed by general theorems \cite{Ma}. This case is really simple and we can compute it explicitly. The matrix of the endomorphism $\exp(t\eee)$ is:

\[  \left( {\begin{array}{ccccccc}
               1 & -t & t^2 & 0 & 0& 0& 0 \\
               0 & 1 & -2t  & 0 & 0& 0& 0  \\
               0 & 0&  1 & 0 & 0 & 0 & 0 \\
               0 & 0&  0 & 1 & -t  & \frac{t^2}{2} &-\frac{t^3}{6} \\
               0 & 0 & 0 & 0 & 1& -t& \frac{t^2}{2} \\
               0 & 0  & 0 & 0 & 0& 1& -t \\
               0 & 0& 0 &  0 & 0& 0 & 1 \\
\end{array} } \right). \]

If we define $\Gamma = 6\Z e_0 \times_\epsilon (\Z e_1 \times \Z e_2 \times  \dots \times \Z e_7)$, then $G/ \Gamma$ is a compact manifold with $\pi_1(G/ \Gamma) = \Gamma$ which inherits both a balanced and a mixed $\S7$ invariant structure.

Moreover, we claim that $G/\Gamma$ is not diffeomorphic to $\sv \times S^1$ for any $7$-dimensional submanifold $\sv$. Since $b_1(G/\Gamma)=2$, it is sufficient to prove that if a nilmanifold $G'/\Gamma'$ is diffeomorphic to $\sv \times S^1$ then, $b_1(\sv \times S^1)\geq 3$, or equivallently, $b_1(\sv)\geq 2$. This assertion turns out to be true because we can check that $\sv$ is homotopically equivalent to a nilmanifold. On the one hand, $\sv$ is an Eilenberg-MacLance space $K(1,\pi_1(\sv))$, because $G'$ is contractible. On the other hand  a group is isomorphic to a lattice of a nilpotent Lie group if and only if it is nilpotent, torsion-free and finitely generated \cite[Theorem 2.18]{Ra72}. Since $\Gamma' = \pi_1(G'/\Gamma')=\pi_1(\sv)\times \Z$, both $\pi_1(\sv)$ and $\Gamma'$ verify the conditions listed above. Thus, there is a nilmanifold $\sv'$ such that $\pi_1(\sv')=\pi_1(\sv)$, which is an Elienberg-MacLane space $K(1,\pi_1(\sv))$. Therefore, $\sv'$ and $\sv$ have the same homotopy type and $b_1(\sv)=b_1(\sv')\geq 2$, because $\sv'$ is a nilmanifold.

\subsection*{A compact manifold with a parallel and a mixed $\S7$ structure.}

Take the same spinor and basis of $\R^7$ as the previous example. Consider the skew-symmetric endomorphism such that $\eee(e_2)=e_3$, $\eee(e_4)=e_5$ and $\eee(X) = 0$ on $\langle e_2,e_3,e_4,e_5 \rangle ^{\perp}$.
The rank of this matrix is two and it is associated to $(0,1,1)$. Therefore, Theorem \ref{qab} guarantees the existence of a parallel invariant $\S7$ structure and other invariant $\S7$ structure which is mixed.
The matrix of the endomorphism $\exp(t\eee_2)$ in the previous basis is:

\[  \left( {\begin{array}{ccccccc}
               1 & 0& 0 & 0 & 0& 0 & 0\\
               0& \cos (t) & \sin (t) & 0 & 0 & 0& 0 \\
               0& -\sin(t) & \cos (t) & 0  & 0 & 0& 0  \\
               0& 0 & 0&  \cos (t) & -\sin (t) & 0 & 0  \\
               0& 0 & 0&  \sin (t) & \cos (t) & 0 & 0 \\
               0 & 0  & 0 & 0 & 0& 1& 0 \\
               0 & 0& 0 & 0 & 0& 0& 1 \\
\end{array} } \right). \]

If $t\in \pi \Z$, the previous matrix has integers coefficients so that $\gamma = \pi\Z e_0 \times_\epsilon (\Z e_1 \times \Z e_2 \times \dots \times \Z e_7) $ is a subgroup. Moreover, $G/ \Gamma$ is a compact manifold with $\pi_1(G/ \Gamma) = \Gamma$ and inherits from $G$ both a parallel invariant $\S7$ structure and a mixed invariant one.

According to Remark \ref{flat}, this manifold is flat. It is the mapping torus of $\exp(\pi \eee) \colon X \to X$, where $X$ is a $7$ torus. Indeed, since $\exp(\pi\eee)^2=\Id$, the $8$-torus is a $2$-fold connected covering of $G/\Gamma$.

% To check that this manifold is not diffeomorfic to an $8$-torus, we compute $H_1 (G/ \Gamma)$. There is a long exact sequence:
%$$ \xymatrix{ H^2(G/\Gamma) \ar[r] & H_1(X) \ar[r]^{(\Id - f)^*} & H_1(X) \ar[r]^{i^*} &  H_1(G/\Gamma) \ar[r] & 0 } $$
%Note that the cokernel of the map $(Id-f)^*\colon H^1(X) \to H^1(X)$ is $\Z^4$. That is, we have a short exact sequence:
%$$ \xymatrix{ 0 \ar[r] & \Z^4 \ar[r] & \Z^7 \ar[r] & H_1(G/\Gamma) \ar[r] &  0.} $$
%Hence, $H_1 (G/ \Gamma)=\Z ^3$.

\section*{Acknowledgements}
I would like to thank my thesis directors, Giovanni Bazzoni and Vicente Mu\~{n}oz, for useful conversations, advices and encouragement.

The author is supported by a grant from Ministerio de Educaci\' on, Cultura y Deporte, Spain (FPU16/03475).

%\bibliographystyle{plain}

%\bibliography{biblio}

\end{document}